\documentclass[article,12pt]{elsarticle}

\usepackage{amssymb,color}
\usepackage{graphicx}
\usepackage{enumerate}
\usepackage{amsmath,amsopn,amsfonts}
\usepackage[hyphens]{url}
\usepackage{microtype}
\usepackage{tikz-cd} 
\usepackage{algorithm}
\usepackage{algpseudocode}
\usepackage{hyperref}

\usepackage{url}

\usepackage{mathrsfs}

\usepackage{complexity} 

\usepackage{amsthm}

\newtheorem{theorem}{Theorem}
\newtheorem{lemma}{Lemma}
\newtheorem{remark}{Remark}
\newtheorem{definition}{Definition}

\newtheorem{corollary}{Corollary}
\newtheorem{proposition}{Proposition}

\usepackage{tikz}
\usetikzlibrary{patterns}

\newcommand\distfun{Cayley distance function}
\newcommand\ClinearC{$\mathcal{C}$--Cayley p.f.~linear--time computable}
\newcommand\linearC{Cayley p.f.~linear--time computable}
\newcommand\CpolyC{$\mathcal{C}$--Cayley polynomial--time computable}
\newcommand\polyC{Cayley polynomial--time computable}

\begin{document}


\begin{frontmatter}

\title{Cayley Polynomial--Time Computable Groups}

\author[label1,label2]{Dmitry Berdinsky}
\ead{berdinsky@gmail.com}

\author[label3]{Murray Elder} 
\ead{murray.elder@uts.edu.au}

\author[label1]{Prohrak Kruengthomya}
\ead{prohrakju@gmail.com}

\address[label1]{Department of Mathematics, Faculty of Science,
 Mahidol University, Bangkok, 10400, Thailand}

\address[label2]{Centre of Excellence in Mathematics, Commission on Higher Education, Bangkok, 10400, Thailand}

\address[label3]{School of Mathematical and Physical Sciences, University of Technology Sydney, Ultimo, NSW 2007, Australia}




\begin{abstract}

We propose a new generalisation of Cayley automatic groups, varying the time complexity of computing multiplication, and language complexity of the normal form representatives.
We first consider groups which have normal form language in the class $\mathcal C$ and multiplication by generators computable in linear time on a certain restricted Turing machine model (position--faithful one--tape).
We show that many of the algorithmic properties of automatic groups are preserved (quadratic time word problem), prove various closure properties, and show that the class is quite large; for example it includes all virtually polycyclic groups.
We then generalise to  groups which have normal form language in the class $\mathcal C$ and multiplication by generators computable in polynomial time on a (standard) Turing machine. Of particular interest is when $\mathcal C=\REG$  (the class of regular languages). We prove that 
$\REG$–Cayley polynomial–time
computable groups include all finitely generated nilpotent groups, the wreath
product $\mathbb Z_2 \wr \mathbb Z^2$, and Thompson’s group $F$.


\end{abstract}


\end{frontmatter}

\section{Introduction}

How one can represent elements of an infinite 
finitely generated group $G$? 
A natural way to do this is to assign for each 
group element $g \in G$ a unique normal form 
which is a string over some finite alphabet 
(not necessarily a generating set).
Kharlampovich, Khoussainov and Miasnikov 
used this approach to introduce the notion 
of a Cayley automatic group \cite{KKM11} 
that naturally extends     
the classical notion of an automatic group 
introduced by Thurston and others \cite{Epsteinbook}. 
They require that the language of normal forms 
to be regular, and that for each $s$ from
some finite set of semigroup generators 
$S \subset G$ there is a 
two--tape synchronous automaton     
recognizing all pairs of strings $(u,v)$  
for which $u$ is the normal form of some 
group element $g$ and $v$ is the normal form 
of the group element $gs$.  
Case, Jain, Seah and Stephan showed that
this   is  equivalent to the existence 
of a {position--faithful}\footnote{see Definition~\ref{defn:posFaith}}
one--tape Turing machine for each $s \in S$ 
which computes the output $v$ from the 
input $u$ in linear time \cite{Stephan_lmcs_13}.  
Is it possible to 
extend the notion of a Cayley automatic 
group which admit normal forms satisfying this (linear time) property?
Can it be extended further requiring not linear but polynomial time?

In this paper we consider groups which 
admit normal forms from some formal language class (not necessarily regular) where multiplication  satisfies the (linear time)  
and (polynomial time) properties. 
We study their algorithmic
and closure properties.  
We analyse examples of such groups and 
their normal forms. 
Furthermore, we study the characterization of these normal 
forms in terms of the \emph{\distfun}\ 
 (as defined in \cite{BT_LATA18} and studied in \cite{eastwest19,BET19}).  
In particular, we investigate examples 
of non--automatic groups for which this  function can be 
diminished to the zero function for some normal forms 
satisfying either (linear time) or (polynomial time) 
properties. This is quite different to the situation for
Cayley automatic representations of these groups, where 
the \distfun\  
is always separated from the 
zero function by some unbounded 
nondecreasing function which depends only on the  group.

{\it Contribution and paper outline.} 
In Section \ref{Cayley_automatic_section} we recall 
the notion of  
a Cayley automatic group 
and a Cayley automatic representation.   

In Section \ref{quasi-Cayley-automatic-section} 
we introduce the notion of a 
\emph{ $\mathcal{C}$--Cayley 
	 position--faithful (p.f.) 
	 linear--time 
	 computable  
 group} and 
a \emph{\ClinearC\
representation},
for a given class of languages $\mathcal{C}$, where we require a language of normal forms to 
be in the class $\mathcal{C}$ and 
 right multiplication by each semigroup 
generator to be computed by a 
position--faithful  one--tape Turing machine 
in linear time. 

We show that  
\ClinearC\
groups preserve some key properties of 
Cayley automatic groups. 
In Theorem 
\ref{quasigeodesicnormalformthm} we show that 
each \ClinearC\
representation has quasigeodesic normal form.  
In Theorem \ref{quadratic_alg_theorem1} we show that 
there is a quadratic time algorithm  
computing this normal form. The latter implies that
for every
\ClinearC\  
group the word problem is decidable in quadratic time, 
see Corollary~\ref{quasiCayleywordproblem}.
In Theorems \ref{finiteextensionthm}, \ref{directproducttheorem} and 
\ref{freeproducttheorem}
we prove that under some very mild restrictions on the class $\mathcal{C}$, the family of \ClinearC\
 groups is closed under taking 
 finite extension, direct product
and  free product, respectively.
Furthermore, in Theorem \ref{CquasiCayley_are_Cgraphautomatic}
we show that (under similar mild restrictions) the family of \ClinearC\ groups is contained in the family of $\mathcal{C}$--graph automatic 
groups introduced by the second author and Taback. 
The collection of all 
\ClinearC\
groups for all classes $\mathcal{C}$ 
forms the family of \emph{\linearC\ groups}.  
In Theorem \ref{subgroupsofquasiCayleyarequasiCayley} 
we show that the family of \linearC\ 
groups
is closed under taking 
finitely generated (f.g.) subgroups. In Theorem 
\ref{L_recursivelyenumerable_quasiautomatic} 
we prove that for each 
\linearC\ 
representation the language of normal forms 
must be recursively enumerable. Moreover, 
in Proposition~\ref{Mikhailova_example} we give 
an example of a \linearC\ 
representation for which the language of
normal forms is not recursive.  
In Theorem~\ref{fgsubgroupglnq} we show that 
the family of \linearC\ groups comprises 
all f.g. subgroups of $\mathrm{GL}(n,\mathbb{Q})$; 
in particular, it includes all polycyclic groups.

In Section \ref{Cayleypoltimecomp_section} 
we consider further generalization of 
\linearC\ 
 groups --   
the notion of a \emph{\CpolyC\
group} and a 
\emph{\CpolyC\
representation}, 
for a given class of languages $\mathcal{C}$, where 
we require a language of normal forms to be in the 
class $\mathcal{C}$ and the right multiplication 
by each semigroup generator to be computed 
by a 
one--tape Turing machine
in polynomial time.  
We note that a $\mathcal{C}$--Cayley 
polynomial--time computable representation 
does not necessary have quasigeodesic normal 
form (in contrast  to the p.f. linear-time case). 
However, assuming that a 
$\mathcal{C}$--Cayley polynomial--time 
computable representation has quasigeodesic 
normal form, in Theorem \ref{normalform_poltime} 
we show that there is a 
polynomial--time algorithm computing this normal 
form. The latter implies that the word problem is 
decidable in polynomial time, see Corollary \ref{wordproblempoltime_cor}. 
In Theorem \ref{Finite_extension_DP_FP_CPTC_thm}  
we prove that, similarly to 
\ClinearC\ 
groups, the families of 
$\mathcal{C}$--Cayley polynomial--time 
computable groups and the ones 
with quasigeodesic 
normal forms are closed 
under taking a finite extension, direct product 
and free product. 
The collection of all $\mathcal{C}$--Cayley 
polynomial--time computable groups for all classes
$\mathcal{C}$ forms the family of \emph{Cayley 
polynomial--time computable groups}. 
In Theorem \ref{fgsubgroups_caypoltime}
we show that the family of Cayley 
polynomial--time computable groups 
and the ones 
with quasigeodesic 
normal forms  are each
closed under taking f.g.
subgroups. 
In the end of Section \ref{Cayleypoltimecomp_section}
we show that the class of 
$\REG$--\polyC\ groups comprises all 
f.g. nilpotent groups, 
where $\REG$ is the class of regular languages. 
Moreover, it includes  examples 
such as the wreath product 
$\mathbb{Z}_2 \wr \mathbb{Z}^2$ and Thompson's 
group $F$.

%
%
%
%
%
%
%
%
%
%
%
%
%
%
%
%
%
%
%

In Section \ref{characterization_section} we 
study the Cayley distance function 
for \linearC\ and $\REG$--Cayley 
polynomial--time computable representations. 
We demonstrate that some properties of 
the Cayley distance function 
which hold for Cayley automatic 
representations (shown in the previous works
\cite{BT_LATA18,eastwest19}) do not hold 
neither for \linearC\ and nor for 
$\REG$--\polyC\ representations.    
Section \ref{sec_conclusion} concludes the paper.
Figure \ref{Venndiag} shows a Venn diagram 
for different classes of groups considered in this paper.

\begin{figure}[h!]
	\centering
	\begin{tikzpicture}[scale=1.1,every node/.style={scale=.9}]

	\draw[rounded corners=10pt]
	(.6,0.2) rectangle ++(2,1);
	
	\draw[rounded corners=10pt]
	(0,0) rectangle ++(6.8,2.5);
	
	\draw[rounded corners=10pt]
	(0,0) rectangle ++(7.8,4);
	
	\draw[rounded corners=10pt]
	(0,0) rectangle ++(8.8,5.5);
	
	\draw[rounded corners=10pt]
	(0,0) rectangle ++(9.8,7);
	
	\draw[dashed, rounded corners=10pt]
	(0,0) rectangle ++(6.8,8.5);

	\node [at = {(1.6,.7)}] {automatic};
	\node [at = {(3.3,1.8)}] {\begin{tabular}{l}Cayley automatic$=$\\\REG--Cayley p.f. linear--time computable\end{tabular}};
	\node [at = {(3.05,3.2)}] {$\mathcal{C}$--Cayley p.f. linear--time computable};
	\node [at = {(2.66,4.7)}] {\begin{tabular}{l}$\mathcal{C}$--Cayley 
		polynomial--time with\\quasigeodesic normal form\end{tabular}};
	\node [at = {(3.2,6.2)}] {$\mathcal{C}$--Cayley 
		polynomial--time computable};
	\node [at = {(3.4,7.7)}] {\REG--Cayley 
		   polynomial--time computable};
	
	\end{tikzpicture}
\label{Venndiag}
\caption{A Venn diagram of classes of interest.}
\end{figure}
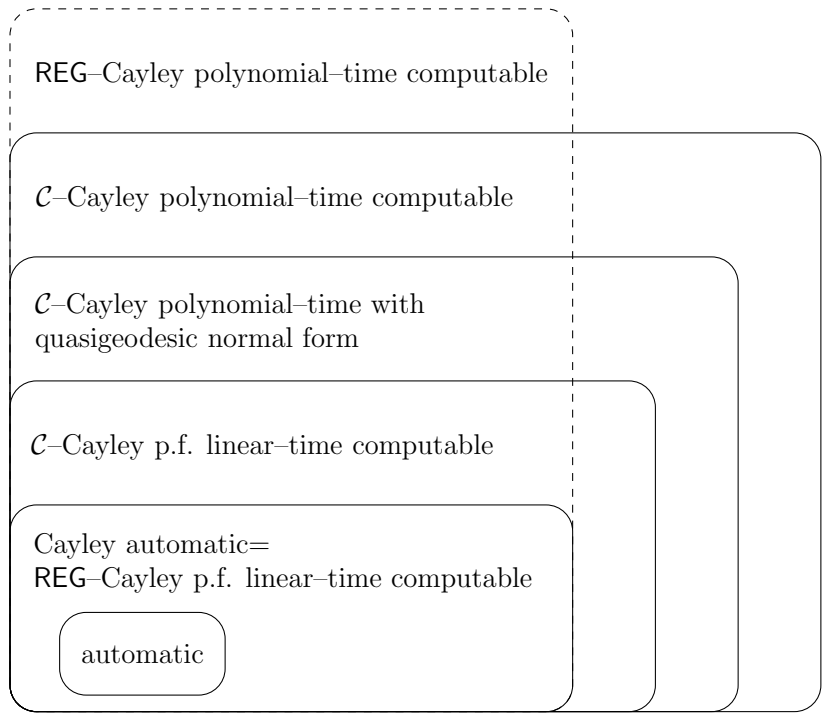

{\it Related work.}   
We briefly mention some previous works which extend the 
notion of an automatic group. 
The motivation to introduce such extensions 
was principally to include all 
fundamental groups of
compact $3$--manifolds.     
Bridson and Gilman introduced the notion 
of an asynchronously $\mathcal{A}$--combable 
group for an arbitrary class of languages 
$\mathcal{A}$ \cite{BridsonGilman1996}.     
Baumslag, Shapiro and Short introduced the 
class of parallel poly--pushdown groups \cite{baumslagshapiro}.      
Brittenham, Hermiller and Holt 
introduced the notion of an autostackable 
group \cite{BrittenhamHermillerHolt14}. 
Kharlampovich, Khoussainov and Miasnikov
introduced the notion of a Cayley 
automatic group \cite{KKM11} from which 
the present paper developed.   
The second author and Taback extended 
the notion of a Cayley automatic group 
replacing the class of regular languages with more 
powerful language classes \cite{ElderTabackCgraph}, which we refine here.

\section{Cayley Automatic Groups}
\label{Cayley_automatic_section}

Kharlampovich, Khoussainov and Miasnikov 
introduced the notion of a Cayley automatic 
group\footnote{A Cayley automatic group is also referred 
	to as a Cayley graph automatic or graph automatic group
	in the literature.}     
\cite{KKM11} as a natural generalization of
the notion of an automatic group \cite{Epsteinbook} 
which uses the same computational model --  
a two--tape synchronous automaton.   
The class of Cayley automatic groups not only comprises 
all automatic groups, but it includes a rich family 
of groups which are not automatic. In particular, 
it includes all f.g. nilpotent groups of nilpotency class two \cite{KKM11}, 
the Baumslag--Solitar groups \cite{dlt14},  
higher rank lamplighter groups \cite{Taback18} and 
all wreath products of the form $G \wr H$, where 
$G$ is Cayley automatic and $H$ is virtually 
infinite cyclic
\cite{wreath_BET21}. We assume that the reader is 
familiar with the notion of a regular language, 
a finite automaton and a multi--tape synchronous 
automaton. Below we briefly recall both 
definitions: for automatic and Cayley automatic groups.  

Let $G$ be a finitely generated   group 
with a finite generating set $A \subset G$.  
We denote by $A^{-1}$ the set of the inverses  
of elements of $A$. 
Let $S = A \cup A^{-1}$. 
For a given word 
$w  = s_1 \dots s_m \in S^*$
let $\pi (w)$ be  the product of elements 
$s_1 \dots s_m$
in the group $G$; if $w$ is the empty string $w = \epsilon$, then $\pi(w)$ is the identity of the group $G$. 
For a given language $L \subseteq S^*$, we denote 
by $\pi : L \rightarrow G$ the canonical map
which sends a string\footnote{We use the terms ``string" and ``word" interchangeably.}  
$w \in L$ to the group element $\pi(w) \in G$.  

It is said that the group $G$ is automatic, if there is a 
regular language $L \subseteq S^*$ such that 
the canonical map $\pi : L \rightarrow G$ 
is bijective and for every $s \in S$ the relation 
$L_s = \{(u,v) \in L \times L \, | \, 
\pi(u)s = \pi(v)\}$ is recognized by a two--tape
synchronous automaton. 
A string $w \in L$ is called a normal form 
for the group element $\pi (w) \in G$;  
accordingly, $L$ is called a language of normal 
forms. We call the bijection $\pi : L \rightarrow G$
an automatic representation of the group $G$. 

Let $\Sigma$ be an arbitrary finite
alphabet.
It is said that the group $G$ is Cayley automatic
if there is a regular language 
$L \subseteq \Sigma^*$ and a bijection 
$\psi: L \rightarrow G$ such that 
for every $s \in S$ the relation 
$R_s = 
\{ (u,v) \in L \times L \, | \, 
\psi(u) s = \psi (v)\}$
is recognized by a two--tape synchronous automaton.  
Similarly, we say that 
$L$ is a language of normal forms and 
a string $w \in L$ is a normal form for the group 
element $\psi (w) \in G$. We call the bijection 
$\psi : L \rightarrow G$ a Cayley automatic 
representation of the group $G$. 
We note that the notion of a Cayley automatic group does not 
require the bijection $\psi : L \rightarrow G$ 
to be canonical. As long as for every $s \in S$ 
the relation $R_s$ is recognized by
a two--tape synchronous automaton, $\psi$ can 
be an arbitrary bijection. Similarly, as long
as $L$ is regular, it can be a language over an 
arbitrary alphabet $\Sigma$.

It is said that a function 
$f : \Sigma^* \rightarrow \Sigma^*$ 
is automatic if the relation 
$R_f = \{(w, f(w)) \in \Sigma^* \times \Sigma^* 
\, | \,  w \in \Sigma^* \}$ is recognized 
by a two--tape synchronous automaton.
So one can equivalently define Cayley 
automatic groups in the following way.

\begin{definition}[Cayley automatic groups]
	\label{Cayley_automatic_def}	
	We say that the group $G$ is Cayley automatic 
	if there exists a regular language 
	$L \subseteq \Sigma^*$  
	over some finite alphabet $\Sigma$,
	a bijective mapping $\psi : L \rightarrow G$ 
	and automatic functions 
	$f_s : \Sigma^* \rightarrow \Sigma^*$, $s \in S$,
	such that: 
	\begin{itemize} 
		\item{$f_s (L) \subseteq L$, that is,  
			$f_s$ maps a normal form to a normal form;}
		\item{for every $w \in L$: 
			$\psi (f_s (w)) = \psi (w) s$, that is,  
			the following diagram commutes:
			\[ \begin{tikzcd}
			L \arrow{r}{f_s} \arrow[swap]{d}{\psi} & 
			L \arrow{d}{\psi} \\%
			G \arrow{r}{\times s}& G
			\end{tikzcd}
			\]}  
	\end{itemize} 
	for all $s \in S$. We call 
	$\psi: L \rightarrow G$ a Cayley automatic representation 
	of $G$. 
\end{definition}	

\begin{remark} 
	The original motivation 
	to study Cayley automatic groups 
	stemmed not only from the notion of 
	an automatic group \cite{Epsteinbook}, but also
	from the notion of a FA--presentable
	structure \cite{KhoussainovNerode95}. 
	Namely, a f.g. group is Cayley automatic if and 
	only if its labelled directed Cayley graph 
	is a FA--presentable structure \cite{KKM11}. 
	For a recent survey of the theory of
	FA--presentable structures we refer the
	reader to \cite{Stephan2015}.           
\end{remark}

\section{Cayley position--faithful linear--time computable 
groups}
\label{quasi-Cayley-automatic-section} 

The notion of a Cayley automatic group 
can be naturally extended further to that of a
Cayley position--faithful linear--time computable 
group which we introduce in this section.

Let us first recall the notion 
of a position--faithful 
one--tape Turing machine (as defined in \cite[p.~4]{Stephan_lmcs_13}). \begin{definition}[Position--faithful one--tape Turing machine]\label{defn:posFaith}
A position--faithful one--tape Turing machine is a Turing machine which uses a 
semi-infinite tape (infinite in one direction only)
with the left--most position containing the special symbol $\boxplus$ which only occurs at this position and cannot be modified. 
The initial configuration of the tape is $\boxplus x \boxdot^\infty$, where $\boxdot$ is a special blank symbol, and $x\in \Sigma^*$ for some alphabet $\Sigma$ with $\Sigma\cap\{\boxplus ,\boxdot\}=\varnothing$. During the computation the Turing machine operates as usual, reading and writing cells to the right of the $\boxplus$ symbol.

A function $f : \Sigma^* \rightarrow \Sigma^*$ is said to be
computed by a position--faithful one--tape Turing machine, if when started with tape content being 
$\boxplus x \boxdot^\infty$, the head initially being at 
$\boxplus$, the Turing machine eventually reaches an accepting state (and halts), with the tape having prefix 
$\boxplus f(x) \boxdot$ where $x,f(x)\in \Sigma^*$. There is no restriction on the output beyond the first appearance of $\boxdot$.  \end{definition}

Case, Jain, Seah and Stephan established the equivalence
of the following classes of functions \cite{Stephan_lmcs_13}:
\begin{itemize} 
	\item {automatic functions 
		$f: \Sigma^* \rightarrow \Sigma^*$;} 
	\item{functions $f: \Sigma^* \rightarrow \Sigma^*$ 
		computed in linear time by a 
		deterministic position--faithful one--tape Turing machine.}  
				\item{functions  $f: \Sigma^* \rightarrow \Sigma^*$  computed in linear time 
			by a nondeterministic position--faithful
			one--tape Turing machine.}   
\end{itemize} 
We say that a function 
$f: \Sigma^* \rightarrow \Sigma^*$ is 
position--faithful (p.f.) linear--time computable if it is computed 
by a (deterministic) position--faithful 
one--tape Turing machine in linear time. 
By the equivalence above 
$f : \Sigma^* \rightarrow \Sigma^*$
is  p.f. linear--time computable if and only if 
it is automatic. So we may use the terms  
automatic function and p.f. linear--time computable function interchangeably.

We note that the requirements 
of being one--tape and  
position--faithful matter.
Consider the following example 
from \cite[p.~4]{Stephan_lmcs_13}: a function which takes input $w\in \Sigma$ and outputs the binary string $v\in\{0,1\}^*$ where $w=uvxy$ with $u\in (\Sigma\setminus\{0,1\})^*$, $x\in\Sigma\setminus\{0,1\}$. An ordinary semi--infinite tape Turing machine can easily compute this: simply move the read head to the first occurrence of $0,1$ on the tape (or replace all cells $u,x,y$ by blank symbols, depending on the Turing machine model). The position--faithful model is not able to perform this function in linear time: it would have to somehow copy the contents of the cells containing $v$ forwards so that they start after the $\boxplus$ symbol, but this would involve at least $O(|u|^2)$ steps.
Note that this 
function can be computed by a 
deterministic position--faithful {\em two--tape} Turing 
machine in linear time.     
The functions
computed by position--faithful one--tape Turing machines form a natural subclass of the class 
of linear--time computable functions. 

In order to make our first generalisation of Cayley automatic, we  allow the language of normal forms to be in any formal language class, not necessarily regular. However we maintain the requirement that right multiplication
by a generator, or its inverse, is  computed by
an automatic function.     
%
Let $G$ be a f.g. group and
$S = \{s_1, \dots, s_n\} \subset G$ be  
a finite set of semigroup generators of $G$: that is, every 
$g \in G$ can be represented as a product of 
elements from $S$. 
Let $\mathcal{C}$ be a nonempty class of 
languages. 
\begin{definition}[$\mathcal{C}$--Cayley 
	 position--faithful
	 linear--time 
	 computable group] 
	\label{quasiautdef1}  
	We say that  $G$  
	is a $\mathcal{C}$--Cayley 
	position--faithful (p.f.)
	linear--time 
	computable group
	if 	there exist 
	a language $L \subseteq \Sigma^*$ from the class 
	$\mathcal{C}$ over some 
	finite alphabet $\Sigma$,  
	a bijective mapping 
	$\psi : L \rightarrow G$ between the language 
	$L$ and the group $G$ and p.f. linear--time 
	computable functions
	$f_i: \Sigma^* \rightarrow \Sigma^*$ 
	such that $f_i (L) \subseteq L$ and
	for every $w \in L$:
	$\psi(f_i (w)) = \psi(w)s_i$, 
	for all $i=1,\dots,n$. We call 
	$\psi : L \rightarrow G$ a \ClinearC\ representation of the group $G$.  
	If the requirement for $L$ to be in a specific 
	class $\mathcal{C}$ is omitted, 
	then we say that $G$ is a \linearC\ group and  $\psi : L \rightarrow G$
	is a \linearC\ representation of $G$. 
\end{definition} 
The class of
$\REG$--\linearC\ groups is simply the class of Cayley automatic groups. Below we show that, similarly to Cayley 
automatic groups, Definition \ref{quasiautdef1} 
does not depend on the choice of generators.  
\begin{proposition}
	\label{generator_independence}   
	The notion of a \ClinearC\ group does not depend on the choice of semigroup generators.
\end{proposition}
\begin{proof} 
	Let $G$ be a \ClinearC\ 
	group for a set of semigroup generators  
	$S= \{ s_1, \dots, s_n\}$. Then there is 
	a language $L \subseteq \Sigma ^*$ from the 
	class $\mathcal{C}$, a bijective mapping 
	$\psi: L \rightarrow G$ and 
	automatic functions 
	$f_i: \Sigma^* \rightarrow \Sigma^*$ 
	such that $f_i (L) \subseteq L$ and 
	$\psi (f_i (w)) = \psi(w) s_i$ for 
	all $i=1,\dots,n$ and $w \in L$. 
	Let $S' = \{s_1', \dots , s_k'\} 
	\subseteq G$ be another set of 
	semigroup generators of the group $G$.
	Each element $s' \in S'$ is a product 
	of elements from $S$. 
	Therefore, for a given $j = 1, \dots, k$ 
	there exist $s_{j_1}, \dots, s_{j_m} \in S$
	for which $s_j ' = s_{j_1}  \dots s_{j_m}$. 
	We define $f_j ' : \Sigma^* \rightarrow \Sigma^*$ 
	to be the composition: 
	$f_j ' = f_{j_m} \circ \dots \circ f_{j_1}$.
	For every $j = 1,\dots,k$ the function $f_j'$ is  automatic, 
	$f_j ' (L) \subseteq  L$ and 
	$\psi (f_j '(w)) = \psi (w) s_j '$ for all $w \in L$. 
	This shows that the definition of the class 
	of \ClinearC\ groups 
	does not depend on the choice of semigroup 
	generators $S$. 
\end{proof}
\begin{remark} 
\label{remark_symmetric_set}	
	By Proposition \ref{generator_independence}, one 
	can always assume that a set of semigroup 
	generators $S$ is symmetric, where 
	the term symmetric means that if $s \in S$, then 
	$s^{-1} \in S$. That is, 
	$S = A \cup A^{-1}$ for some finite 
	set of generators $A$ of the group $G$, where
	$A^{-1} = \{a^{-1} \, | \, a \in A\}$.
	Note that $A$ may or may not include the identity of $G$.        	 
\end{remark}	
\begin{remark}
	Similarly to Cayley automatic groups, 
	\linearC\ groups are related 
	to the notion of a FA--presentable structure.  
	Let $\mathfrak{B}$ be the structure 
	$\mathfrak{B}=\left(\Sigma^*; 
	\mathrm{Graph}(f_1),\dots,\mathrm{Graph}(f_n) \right)$ 
	for some $\Sigma$ and $f_1,\dots,f_n$ from 
	Definition \ref{quasiautdef1}, 
	where $\mathrm{Graph}(f)$ for a function 
	$f: \Sigma^* \rightarrow \Sigma^*$ is the 
	binary relation
	$\mathrm{Graph}(f) = \{(w,f(w)) \in 
	\Sigma^* \times \Sigma^* \, | \, w \in \Sigma^*\}$. Then every 
	structure $\mathfrak{B}' = 
	\left(B'; f_1',\dots,f_n'\right)$ isomorphic 
	to the structure $\mathfrak{B}$ is FA--presentable. 
	Let $\Gamma$ be the directed labelled graph
	$\Gamma = \left(G; E_1,\dots, E_n\right)$, 
	where $E_i = \{(g_1,g_2) \in G \times G \,|\, 
	g_1 s_i = g_2\}$. 
	Then the bijection $\psi^{-1} : G \rightarrow L$ is 
	an embedding of the structure $\Gamma$ into 
	the structure $\mathfrak{B}$.       
\end{remark}

\subsection{Quasigeodesic Normal Form}

We notice that the analogue of the bounded  
difference lemma 
(see \cite[Lemma~2.3.9]{Epsteinbook} for automatic 
and \cite[Lemma~8.1]{KKM11} for Cayley automatic 
groups) holds for \linearC\ 
groups as well.
Let $G$ be a \linearC\ group 
and $\psi : L \rightarrow G$ 
be a \linearC\ representation of $G$
for some language $L \subseteq \Sigma^*$. 
Let $S = A \cup A^{-1}$ for  
some finite set of generators $A$ of the group 
$G$, see Remark \ref{remark_symmetric_set}.  
\begin{lemma} 
	\label{bounded_difference_lemma}
	There exists a constant $K>0$ such that  
	for every $g \in G$ and
	$s \in S$,
	if $u,v \in L$ are the strings representing 
	$g$ and $gs$, respectively: $\psi(u)=g $ 
	and $\psi(v)=gs$, then  
	$||u|-|v|| \leqslant K$.
\end{lemma} 	
\begin{proof} 
	For every 
	$s \in S$
	there is 
	an automatic function  
	$f_s : \Sigma^* \rightarrow \Sigma^*$ 
	such that $f_s (u) = v$ for all  
	$u,v \in L$  for which 
	$\psi (u)s = \psi (v)$ in the group $G$. 
	For a given 
	$s \in S$, 
	let $M_s$ be a (nondeterministic) two--tape 
	synchronous automaton recognizing the 
	relation 
	$R_{f_s} = \{(w, f_s(w)) \in \Sigma^* \times \Sigma^* 
	\, | \,  w \in \Sigma^*\}$. 
	By the pumping lemma, for every 
	$(u,v) \in R_{f_s}$ the following 
	inequality holds:  
	$$
	|v| \leqslant |u| + N_s,   
	$$
	where $N_s$ is the number of states 
	of the automaton $M_s$. 
	Therefore, 
	for all $u,v \in L$, if 
	$\psi(u)=g$ and $\psi(v)=gs$ 
	for some $g \in G$ and 
	$s \in S$, 
	then $||u|-|v|| \leqslant K$, 
	where 
    $K = \max \{ N_s \, | \, s \in S\}$
\end{proof}	
For a given group element $g \in G$ we denote by 
$d_A (g)$ the length of a geodesic word representing 
$g$ with respect to the set of generators $A$.

\begin{definition}[\cite{ElderTabackCgraph}] 
\label{quasigeodesicnf_def} 	
	Let $\psi : L \rightarrow G$ be a 
	 bijection between 
	some language $L \subset \Sigma^*$ and 
	$G$. 
	It is said that a representation 
	$\psi : L \rightarrow G$ has quasigeodesic
	normal form if there is a constant $C$ 
	such that for all $w \in L$: 
	$|w| \leqslant C \left(d_A (\psi(w))+1 \right)$. 
\end{definition}

\begin{theorem}[Quasigeodesic normal form]  
	\label{quasigeodesicnormalformthm}   
	A \linearC\ representation $\psi : L \rightarrow G$
	has quasigeodesic normal form.  
\end{theorem}
\begin{proof} 
	For a given $w \in L$, let $s_1  \dots  s_n$,
	for $s_i \in A \cup A^{-1}, i=1,\dots,n$,
	be a geodesic in $G$ with 
	respect to the set of generators $A$ such that 
	$s_1 \dots s_n = \psi(w)$ in $G$; so, $d_A(\psi (w)) =n$.  
	We denote by $w_0$ the string representing  
	the identity $e$: $\psi(w_0) = e$. For a
	given $i \in \{1,\dots,n\}$, let  
	$w_i =\psi^{-1}( s_1 \dots s_i)$.     
	By Lemma \ref{bounded_difference_lemma}, 
	$|w_{i+1}| \leqslant |w_i| + K$ for all 
	$i=0,\dots,n-1$ and some constant $K$. 
	Therefore, $|w| \leqslant n K + |w_0|$.
	Let $C = \max \{K, |w_0|\}$. Thus, 
	$|w| \leqslant C (d_A (\psi(w)) + 1)$. 
\end{proof}
\subsection{Algorithmic Properties}

A key property of automatic and Cayley automatic groups 
 is 
the existence of a quadratic time algorithm 
which for a given word $v \in
\left( A \cup A^{-1} \right)^*$ 
finds the normal form $u \in L$, i.e., the  
string for which
$\psi(u) = \pi(v)$;
see \cite[Theorem~2.3.10]{Epsteinbook} 
and \cite[Theorem~8.2]{KKM11} for 
automatic and Cayley automatic groups, 
respectively. Below we show that this 
property holds  for \linearC\ 
groups as well.    
\begin{theorem}[Computing normal form in quadratic time]
	\label{quadratic_alg_theorem1} 	
	There is an algorithm which for a given input word 
	$v \in (A \cup A^{-1})^*$ 
	computes the string $u \in  L$, 
	for which $\psi(u) = \pi (v)$ in the group $G$.
	Moreover, this algorithm can be implemented  
	by a deterministic position--faithful one--tape 
	Turing machine in quadratic time.  
\end{theorem}		
\begin{proof}    
	Let us be given   
	the string $u_0 \in L$ 
	representing the identity
	$e \in G$: $\psi(u_0)=e$. 
	Let $v = s_1 \dots s_k$, where $s_i \in A \cup A^{-1}$.
	For a given $i=1,\dots,k$ we denote by 
	$TM_{s_i}$ a position--faithful  
	deterministic one--tape Turing machine which computes 
	the function $f_{s_i}$ in linear time. 
    To simplify the exposition let us assume 
    first that there are two  tapes. Initially, the configuration of the first and the second tapes are   
    $\boxplus v \boxdot^\infty$ and 
    $\boxplus \boxdot^\infty$, respectively, with 
    the heads over the special symbol $\boxplus$.   
	A general outline of an algorithm computing the representative 
	string $u \in L$ for the input string $v$ is as follows.
    
    First the algorithm writes the string $u_0$ on the
    second tape and moves the head back to the initial position. Then on the first tape it makes one 
    move to the right, reads off the 
    symbol $s_1$, marks the cell 
    and moves the head back to 
    the initial position. 
    Then on the second tape it 
    computes the representative $u_1 \in L$ 
    of $s_1$ by feeding $u_0$ to $TM_{s_1}$ as the input 
    and moves the head back to the initial position. 
    Then on the first tape the head moves to the right until it encounters the first non--marked symbol $s_2$,  reads it off,  marks the cell and moves 
    the head back to the initial position. 
    Then on the second tape it computes 
    the representative $u_2 \in L$  of $s_1s_2$ by feeding $u_1$ to $TM_{s_2}$ as the input and moves the head back to the initial position. 	
  	Continuing in this way it finally computes the  
  	representative $u_k \in L$ of the group 
  	element $s_1 \dots s_k$. 
	    
	By Lemma \ref{bounded_difference_lemma}, 
	for every $u_j \in L$ representing 
	the element $s_1 \dots s_j$ we have 
	$|u_j| \leqslant  |u_{0}| + K j$, $j =1, \dots,k$.  
	Moreover,  there are constants $C_1,C_0$ such that 
	for every $j =1, \dots,k$ 
	the Turing machine $TM_{s_j}$ computes 
	$u_{j}$ from the input $u_{j-1}$ in time 
	at most $C_1 |u_{j-1}|  + C_0$. 
	Therefore, $TM_{s_j}$ computes 
	$u_{j}$ from the input $u_{j-1}$ in time 
	at most $C_1 K (j-1) + C_1 |u_0|+C_0$ for each 
	$j=1,\dots,k$.       
	Furthermore, $2j$ moves are required to read off 
    a symbol $s_j$ and return the head back to the 
    initial position for each $j=1,\dots,k$.  
	Thus, the total number of 
	moves is at most quadratic.  
	The last thing to note is that the algorithm 
	can be implemented on one tape  
	using $2$--tuple symbols: 
	if a symbol $\beta$ on the first tape appears 
	on top of a symbol $\gamma$ on the second tape, 
	this pair can be encoded by a
	$2$--tuple symbol  
	$\beta \choose \gamma$.	
\end{proof}
\begin{corollary}[Solving the word problem in quadratic time] 
	\label{quasiCayleywordproblem}	
	For a \linearC\ group
	the word problem can be solved  
	by a deterministic  one--tape Turing machine
	in quadratic time.  	 
\end{corollary}
\begin{proof}
	An algorithm solving the word problem in $G$ 
	is as follows.   
	For a given input word $v \in (A \cup A^{-1})^*$  
	it first finds the string $u \in L$ 
	representing $\pi(v)$: $\psi(u) = \pi (v)$, as 
	it is described in Theorem \ref{quadratic_alg_theorem1},  and then compares $u$ with the string $u_0$ representing 
	the identity $e \in G$: if $u = u_0$, then 
	$\pi(v) = e$; otherwise, $\pi(v) \neq e$.  
	This algorithm can be implemented by a deterministic 
	one--tape Turing machine in quadratic time. 
\end{proof}

\begin{theorem} 
	\label{L_recursivelyenumerable_quasiautomatic}	
	Let $\RE$ denote the class of recursively 
	enumerable languages. 
	For every \linearC\ representation 
	$\psi: L \rightarrow  G$ the language 
	$L$ is in the class $\RE$. 
\end{theorem}
\begin{proof}	
	A procedure listing all words of the language 
	$L$ is as follows. It consecutively takes 
	$v \in (A \cup A^{-1})^*$  as the input to produce 
	the output $\psi^{-1}(v) \in L$ using the algorithm 
	described in Theorem \ref{quadratic_alg_theorem1}.
	This procedure lists all strings of the 
	language $L$. 
\end{proof}

\begin{proposition}
	\label{Mikhailova_example}	
	Let $\R$ denote the class recursive languages. The class of $\left(\RE \setminus\R\right)$--\linearC\ groups is non--empty.
\end{proposition}	
\begin{proof}  
	There exists a f.g. subgroup 
	$H \leqslant F_2 \times F_2$ 
	with undecidable membership 
	problem \cite{Mik66}: 
	given a word $w$ over some generating set 
	of $F_2  \times F_2$, decide whether 
	$\pi(w)$ is an element of $H$. 
	Let $\psi : 
	L \rightarrow F_2 \times F_2$ 
	be a \linearC\ representation 
	of $F_2 \times F_2$ (e.g., it can be Cayley 
	automatic or even automatic). 
	Let $L' = \psi^{-1}(H)$ and  
	$\psi': L' \rightarrow H$ be the restriction of $\psi$ onto $L'$: $\psi' = \psi|_{L'}$. 
	By Theorem \ref{subgroupsofquasiCayleyarequasiCayley} below, 
	$\psi' : L' \rightarrow H$ is a \linearC\ 
	representation of $H$. If $L'$ is recursive, then 
	the algorithm solving membership problem for $H$ is as follows. 
	For a given word $w$ over some generating set of 
	$F_2 \times F_2$ we first find the string 
	$u \in L$ for which 
	$\psi (u) = \pi(w)$  in $F_2 \times F_2$ (see   
	the algorithm in Theorem \ref{quadratic_alg_theorem1}) 
	and then verify whether $u$ is in the language $L'$ or not. Therefore, assuming that $L'$ is recursive, we get 
	that the membership problem for the subgroup 
	$H \leqslant F_2 \times F_2$ must be decidable, which 
	leads to a contradiction. 
	Therefore, $L'$ is not a recursive language, although it is recursively enumerable by Theorem \ref{L_recursivelyenumerable_quasiautomatic}.         	
\end{proof}

\subsection{Closure Properties}

Now we turn to closure properties for 
\linearC\ groups.     
Let $C$ be a given class of languages.  
Throughout the paper we assume that   
$\mathcal{C}$ is closed under a change of 
symbols in the alphabet. 
That is, if $\xi: \Sigma \rightarrow \Sigma'$ is 
a bijection  between two finite alphabets $\Sigma$ and 
$\Sigma'$ and $L$ is in the class $\mathcal{C}$, then 
the image of $L$ under the homomorphism 
induced by $\xi$ is also in the class $\mathcal{C}$.    
\begin{theorem}[Finite extensions] 
	\label{finiteextensionthm}
	Assume that a class of languages 
	$\mathcal{C}$ satisfies the 
	following closure property:  
	if $L \subseteq \Sigma^*$ is in the class $\mathcal{C}$ 
	and $L_0$ is a finite language 
	over some $\Sigma_0$ for which $\Sigma \cap \Sigma_0 = \varnothing$, then the concatenation 
	$LL_0$ is in the class $\mathcal{C}$.    
	Then, a finite extension of a 
	\ClinearC\ 
	group is \ClinearC.  
\end{theorem}	
\begin{proof} 
	Let $H$ be a subgroup of finite index of 
	a group $G$. Suppose that $H$ is 
	\ClinearC.    
	Then there exists a 
	\ClinearC\ 
	representation $\psi : L \rightarrow H$ for some 
	language $L \subseteq \Sigma^*$ 
	in the class $\mathcal{C}$.    
	The following is similar to the argument  
	from \cite[Theorem~10.1]{KKM11} 
	which shows that a finite extension 
	of a Cayley automatic group is  Cayley automatic.  
	Every $g \in G$ is uniquely represented 
	as a product $g = h k$, where $h \in H$ and 
	$k \in K$ for some finite subset 
	$K = \{ k_0, k_1, \dots, k_m\}\subset G$ 
	that contains the identity: $k_0 = e \in G$.
	Let $\Sigma_0 = \{\sigma_1, \dots,\sigma_m\}$ for
	some symbols $\sigma_1, \dots,\sigma_m$, 
	which are not in $\Sigma$, and 
	$L_0 \subset \Sigma_0 ^*$ be the finite 
	language $L_0 = \{\epsilon,\sigma_1, \dots,\sigma_m\}$.
	We denote by $L'=LL_0$ the concatenation of 
	the languages $L$ and $L_0$. 
	By the assumption of the theorem, 
	the language $L'$ is in the class $\mathcal{C}$.
	Let $A = \{a_1,\dots, a_n\}$ be some set of 
	generators of $H$. Then $A \cup A^{-1} \cup K$ is 
	a set of semigroup generators for $G$.    
	
	We define a bijection $\psi' : L' \rightarrow G$ 
	as follows. For a given $w' \in L'$, 
	$w'$ is the concatenation:  
	$w' = w u$ for some $w \in L$ and $u \in L_0$.     
	Let $\varphi : L_0 \rightarrow K$ be a bijection 
	between $L_0$ and $K$ for which 
	$\varphi(\epsilon) = k_0, \varphi(\sigma_1)=k_1, \dots, 
	\varphi(\sigma_m) = k_m$. 
	We put $\psi' (w') = \psi(w) \varphi(u)$. 
	The right multiplication of $g \in G$,
	where 
	$g = hk$ is the unique representation
	of $g$, as above, by 
	$q \in A \cup A^{-1} \cup K$ 
	is given by the formula:
	$$ g q = hkq = h s_1 \dots s_\ell k_j, $$   
	for some $s_1, \dots, s_l \in A \cup A^{-1}$ and 
	$k_j \in K$ which depend only on 
	$k$ and $q$: $kq = s_1 \dots s_\ell k_j$. 
	An algorithm transforming the input 
	$\psi'^{-1}(g)$ to the output 
	$\psi'^{-1}(gq)$, implemented by 
	a position--faithful 
	one--tape Turing machine, is as follows. 
	First the head moves to the rightmost cell which 
	contains the symbol $\varphi^{-1}(k)$ 
	(or the blank symbol if $k=e$), reads it off, stores it 
	in the memory and changes it to the blank symbol; 
	then the head moves back to the initial cell. 
    The non--blank portion of the tape now consists of the string $\psi^{-1}(h)$. 
	After that an algorithm computing   
	multiplication by $s_1\dots s_\ell$  in the group 
	$H$ is run; once it is finished, the string 
	$\psi^{-1}(hs_1 \dots  s_\ell)$ is written on the tape.    
	Then the head moves to the first blank symbol 
	to change it to $\sigma_j$, 
	unless $k_j=e$ -- 
	in this case no action is needed, so the machine halts. 
	Then the head moves right to the next cell 
	and changes its content to the blank symbol 
	if it is non--blank.   
	After that 
	the machine halts. Now the string 
	$\psi^{-1}(hs_1\dots s_\ell)\varphi^{-1}(k_j)$, 
	which is equal to $\psi'^{-1}(gq)$,   
	is written on the tape.
	Clearly, at most linear time is required for this
	algorithm. Thus, $G$ is 
	\ClinearC. 
\end{proof} 
\begin{theorem}[Direct products]
	\label{directproducttheorem}	
	Assume that a class of languages $\mathcal{C}$ 
	satisfies the following closure property: 
	if $L_1 \subseteq \Sigma_1 ^*$ 
	and $L_2 \subseteq \Sigma_2 ^*$
	are languages in the class  
	$\mathcal{C}$ for which 
	$\Sigma_1 \cap \Sigma_2 = \varnothing$, 
	then the concatenation $L_1 L_2$ is in 
	the class $\mathcal{C}$.  
	Then, the direct product of two 
	\ClinearC\ groups 
	is \ClinearC. 
\end{theorem}
\begin{proof}
	Let $G_1$ and $G_2$ be two 
	\ClinearC\ 
	groups. Then there exist 
	\ClinearC\  
	representations $\psi_1 : L_1 \rightarrow G_1$ 
	and $\psi_2 : L_2 \rightarrow G_2$ for some 
	languages 
	$L_1 \subseteq \Sigma_1 ^*$ and 
	$L_2 \subseteq \Sigma_2 ^*$ in the 
	class $\mathcal{C}$ for which 
	$\Sigma_1 \cap \Sigma_2 = \varnothing$.     
	We denote by $L$ the concatenation: 
	$L = L_1 L_2$. By the assumption of the 
	theorem the language $L$ is in the class 
	$\mathcal{C}$. Let $A = \{a_1,\dots,a_{n_1}\}$ 
	and $B = \{b_1,\dots,b_{n_2}\}$ be some sets
	of generators for the groups $G_1$ and $G_2$, respectively.  
	Then $A \cup A^{-1} \cup B \cup B^{-1}$ is a set of 
	semigroup generators for the group 
	$G = G_1 \times G_2$. 
	The groups $G_1$ and $G_2$ can be considered  
	as subgroups of $G$. Every group
	element $g \in G$ can be uniquely 
	represented as the product: $g= g_1 g_2$, 
	where $g_1 \in G_1$ and $g_2 \in G_2$.        
	
	Let $L = L_1 L_2$ and $\psi : L \rightarrow G$ be a bijection defined 
	as follows. For a given $w \in L$, let $w$ be the 
	concatenation $w=uv$ for some $u \in L_1$ and 
	$v \in L_2$. We put 
	$\psi (w) = \psi_1(u)\psi_2(v)$. The right 
	multiplication of $g = g_1 g_2$, where $g_1 \in G_1$ 
	and $g_2 \in G_2$, by 
	$q \in A \cup A^{-1} \cup B \cup B^{-1}$ 
	is given by 
	$gq = (g_1 q)g_2$ if $q \in A \cup A^{-1}$ and 
	$gq = g_1 (g_2 q)$ if $q \in B \cup B^{-1}$. 
	For the case $q \in A \cup A^{-1}$, an algorithm 
	transforming the input $\psi^{-1}(g) = 
	\psi_1 ^{-1}(g_1) \psi_2 ^{-1}(g_2)$ to the 
	output $\psi^{-1}(gq)$, implemented by a position--faithful one--tape 
	Turing machine, is as follows. First 
	it transforms the prefix
	$\psi_1^{-1}(g_1)$ to the prefix 
	$\psi_1^{-1}(g_1 q)$.   
	If overlapping with the substring 
	$\psi_2^{-1}(g_2)$ occurs, then it can be encoded by 
	$2$--tuple symbols:  		
	if symbols 
	$\beta$ and $\gamma$ appear on the same cell of the tape,
	then it can be encoded by the
	$2$--tuple symbol 
	$\beta \choose \gamma$.
	After that the 
	algorithm shifts the substring $\psi_2^{-1}(g_2)$ 
	either to the left or to the right, so it is written
	right after the string $\psi_1 ^{-1}(g_1q)$.        
	By Lemma \ref{bounded_difference_lemma}, 
	only shifting (left or right) by a constant 
	number of cells is needed.  Therefore, 
	at most linear time is required for our algorithm. 
	If  $q \in B \cup B^{-1}$, an algorithm just 
	updates the suffix $\psi_2 ^{-1}(g_2)$ 
	to the suffix  $\psi_2 ^{-1}(g_2 q)$ while 
	the prefix $\psi_1 ^{-1}(g_1)$ remains unchanged.   
	Clearly, at most linear time is needed for this 
	algorithm to be implemented by a one--tape 
	position--faithful Turing machine.  
	Finally we conclude that  $G$ is a \ClinearC\ group.                       
\end{proof}

\begin{theorem}[Free products]
	\label{freeproducttheorem}	
	Assume that a class of languages $\mathcal{C}$ 
	satisfies the following closure 
	properties:
	\begin{enumerate}[(a)]
		\item{if a nonempty language $L$,
			for which the empty string 
			$\epsilon \notin L$, is in the class 
			$\mathcal{C}$, 
			then for every $w \in L$ the language 
			$\left(L \setminus \{w\} \right) 
			\vee \{\epsilon\}$} must be in the class 
		$\mathcal{C}$; 
		\item{if $L_1 \subseteq \Sigma_1 ^*$ 
			and $L_2 \subseteq \Sigma_2 ^*$ 
			are languages, 
		    which contain the empty string 
			$\epsilon \in  L_1, L_2$ and
			for which $\Sigma_1 \cap  \Sigma_2 = \varnothing$, 
			are in the class $\mathcal{C}$, 
			then the language 
			$L= (L_1' L_2')^* \vee
			(L_1' L_2')^* L_1' \vee 
			(L_2' L_1')^* \vee         
			(L_2' L_1')^* L_2' 
			\vee \{\epsilon\}$ 
			must be in the class $\mathcal{C}$, where 
			$L_1 ' = L_1 \setminus \{\epsilon\}$
			and $L_2 ' = L_2 \setminus \{\epsilon\}$.}      
	\end{enumerate}
	Then, the free product of two 
	\ClinearC\ groups is \ClinearC. 
\end{theorem}
\begin{proof}       
	Let $G_1,G_2$ be  
	\ClinearC\ groups. 
	There exist \ClinearC\ representations 
	$\psi_1: L_1 \rightarrow G_1$ and 
	$\psi_2: L_2 \rightarrow G_2$ for some languages 
	$L_1 \subseteq \Sigma_1 ^*$ and 
	$L_2 \subseteq \Sigma_2 ^*$ in the class 
	$\mathcal{C}$ for which 
	$\Sigma_1 \cap \Sigma_2 = \varnothing$.   
	Suppose that $\epsilon \in L_1$ and
	for some string $w \in L_1$, 
	$w \not= \epsilon$:  $\psi_1 (w) =e$ in $G_1$. 
	Let $\psi_1' : L_1 \rightarrow G_1$ be a bijective 
	map for which $\psi_1' (u)=\psi_1 (u)$ for all  
	$u \in L_1 \setminus \{\epsilon,w\}$ and 
	$\psi_1'(w)=\psi(\epsilon)$ and 
	$\psi_1'(\epsilon)=e$.
    Let us show that $\psi_1': L_1 \rightarrow G_1$ is a  
    \ClinearC\ representation.
      Let $A = \{ a_1, \dots, a_{n_1}\}$ be a
      set of generators of the group $G_1$. 
      For a given $s \in A \cup A^{-1}$, we denote by
      $f_s : \Sigma_1 ^* 
       \rightarrow \Sigma_1^*$ 
      a p.f. linear--time computable function
      for which $f_s(L_1) \subseteq L_1$ and for 
      every $u \in L_1$: $\psi_1(f_s (u)) =\psi_1(u) s$. 

    Let $v,v'\in L_{1}$ be the strings defined by the identities $f_{s}(v)=\epsilon$ and $f_{s}(v')=w$.
    Let us assume that $v,v'$ are not equal to neither $\epsilon$ nor $w$.
    We define a function $f_{s}' : \Sigma_{1}^{*} \rightarrow\Sigma_{1}^{*}$ as follows.
    Let $f_{s}'(u)=f_{s}(u)$ for all $u\in\Sigma_{1}^{*}\setminus\{\epsilon,w,v,v'\}$ and $f_{s}'(\epsilon)=f_{s}(w)$, $f_{s}'(w)=f_{s}(\epsilon)$, $f_{s}'(v)=w$ and $f_{s}'(v')=\epsilon$.   
    Let us prove that $\psi_{1}'(f_{s}'(u))=\psi_{1}'(u)s$ 
    for all $u \in L_1$. 
    If $u \in L_1$ is not in the set $\{\epsilon,w,v,v'\}$,
    then
    $\psi_{1}'(f'_{s}(u))=\psi_{1}'(f_{s}(u))=\psi_{1}(f_{s}(u))=\psi_{1}(u)s=\psi_{1}'(u)s$.    
    Furthermore, $\psi_{1}'(f'_{s}(\epsilon))=\psi_{1}'(f_{s}(w))=
    \psi_{1}(f_{s}(w))=\psi_{1}(w)s=\psi_{1}'(\epsilon)s$,
    $\psi_{1}'(f'_{s}(w))=\psi_{1}'(f_{s}(\epsilon))=
    \psi_{1}(f_{s}(\epsilon))=\psi_{1}(\epsilon)s
    =\psi_{1}'(w)s$,
    $\psi_{1}'(f_{s}'(v))=\psi_{1}'(w)=\psi_{1}(\epsilon)
    =\psi_{1}(f_{s}(v))=\psi_{1}(v)s=\psi_{1}'(v)s$ 
    and $\psi_{1}'(f_{s}'(v'))=\psi_{1}'(\epsilon)=
    \psi_{1}(w)=\psi_{1}(f_{s}(v'))=\psi_{1}(v')s=
    \psi_{1}'(v')s$.
    
    Let us show that $f_{s}'$ is p.f. linear--time computable by a deterministic one–tape Turing machine. First the algorithm checks whether the input   
    string is $\epsilon,w,v$ or $v'$; if that is so, it writes on the tape $f_{s}(w),f_{s}(\epsilon),w$ or $\epsilon$, respectively, and halts. Clearly, this requires at most constant amount
    of time. If the input string is not $\epsilon,w,v$ or $v'$, the algorithm proceeds as an 
    algorithm for computing $f_{s}$. Therefore, the function $f_{s}'$ is
    p.f. linear--time computable by a deterministic one–tape Turing machine. The cases when $v$ or $v'$ are equal to either $\epsilon$ or $w$ are considered in a similar way.

	Now suppose that $\epsilon \not\in L_1$. Let 
	$w $ be a word from $L_1$ for which $\psi_1(w)=e$. 
	By the property (a), the language 
	$L_1'' = (L_1 \setminus \{w\}) \vee \{\epsilon\}$ 
	is in the class $\mathcal{C}$. 
	Let $\psi_1 '' : L_1'' \rightarrow G_1$ 
	be a bijective map for which 
	$\psi_1 '' (u) = \psi_1(u)$ for all 
	$u \in L_1 \setminus  \{w\}$ and 
	$\psi_1 ''(\epsilon) = e$. 
	By an argument similar to the above 
	it can be shown that  $\psi_1 '': L_1''\rightarrow G_1$ is a \ClinearC\ representation.
	Therefore, we can always assume that 
	$\epsilon \in L_1$ and $\psi_1(\epsilon)=e$  
	in $G_1$ and, similarly,
	$\epsilon \in L_2$ and $\psi_2(\epsilon)=e$ in $G_2$. 
	
	The groups $G_1$ and $G_2$ are naturally embedded  
	in the free product $G = G_1 \star G_2$, so we consider them 
	as the subgroups of $G$. 
	Now let $L= (L_1' L_2')^* \vee
	(L_1' L_2')^* L_1' \vee 
	(L_2' L_1')^* \vee         
	(L_2' L_1')^* L_2' 
	\vee \{\epsilon\}$, where  
	$L_1 ' = L_1 \setminus \{ \epsilon \}$ and 
	$L_2 ' = L_2 \setminus \{ \epsilon \}$. 
	By the property (b), the language $L$ is 
	in the class $\mathcal{C}$. 
	We define a bijection  
	$\psi: L \rightarrow G$ as follows. 
	We put $\psi (\epsilon) = e$. For 
	$w = u_1 v_1 \dots u_n v_n \in (L_1 ' L_2 ')^*$, 
	where $u_i \in L_1 '$ and $v_i \in L_2 '$ for  
	$i=1, \dots, n$, we put:  
	$\psi(w) = \psi_1 (u_1) \psi_2 (v_1) \dots 
	\psi_1 (u_n) \psi_2 (v_n)$.      
	For $w \in (L_1' L_2')^* L_1',  (L_2' L_1')^*$ and  
	$(L_2' L_1')^* L_2'$, $\psi (w)$ is defined in 
	a similar way.  
	
	Let $B = \{b_1, \dots, b_{n_2} \}$ be a set of 
	generators for the group $G_2$. 
	Then $A \cup A^{-1} \cup B \cup B^{-1}$ is a 
	set of semigroup generators for the group 
	$G = G_1 \star G_2$.  
	For a given $g \in G$ let us assume that  
	$\psi^{-1} (g) = u_1 v_1 \dots u_n v_n  
	\in (L_1 ' L_2 ')^*$.     
	If $q \in B \cup B^{-1}$, an algorithm 
	transforming the input $\psi^{-1}(g)$ 
	to the output $\psi^{-1}(gq)$ updates the 
	suffix $v_n$ to the suffix 
	$\psi_2 ^{-1}(\psi_2 (v_n) q)$ while the 
	prefix $u_1 v_1 \dots v_{n-1}u_n$ remains unchanged.    
	If $q \in A \cup A^{-1}$, an algorithm 
	simply attaches the string $\psi_1 ^{-1} (q)$ to 
	$\psi^{-1}(g)$ as a suffix.  
	For $\psi^{-1}(g) \in (L_1' L_2')^* L_1' \vee 
	(L_2' L_1')^* \vee         
	(L_2' L_1')^* L_2'$ an algorithm transforming 
	$\psi^{-1}(g)$ to $\psi^{-1}(gq)$ is 
	realised in a similar way.
	The case $g=e$ is trivial.  
	Clearly, this algorithm can be implemented  
	by a one--tape position--faithful Turing 
	machine in linear time. 
	Thus, the group $G$ is 
	\ClinearC. 
\end{proof}
\begin{remark}
\label{remark_properties_language_freeprod}
   In Theorem \ref{freeproducttheorem},
   closure property (b), the language 
   $L$ is the set of all concatenations 
   $w_1w_2 \dots w_k$ of non--empty strings 
   $w_i \in L_1 \vee L_2$, $i=1,\dots,k$  
   for which none of the consecutive strings 
   $w_i$ and $w_{i+1}$ belong to the same 
   language $L_1$ or $L_2$. 
   The case $k=0$ corresponds to the 
   empty string $\epsilon$.
\end{remark}	
\begin{remark} 
	\label{remarkconditions1}	
	We note that the conditions imposed on the 
	class $\mathcal{C}$ 
	in Theorems \ref{finiteextensionthm}, \ref{directproducttheorem} 
	and \ref{freeproducttheorem}
	are  weak. 
	These conditions 
	are satisfied for many 
	classes of languages including,
	e.g., regular,  
	(deterministic) context--free, 
	(deterministic) context--sensitive, 
	recursive, 
	$k$--counter, $k$--context--free.     
\end{remark}	

\begin{theorem}[Finitely generated subgroups]
	\label{subgroupsofquasiCayleyarequasiCayley}   
	A finitely generated subgroup of a 
	\linearC\ group is \linearC. 
\end{theorem}

\begin{proof} 
	Let $G$ be a \linearC\ group 
	and $S=A \cup A^{-1} = \{s_1,\dots,s_n\}$ be a set of semigroup 
	generators of $G$.   
	Then there is a \linearC\ representation  
	$\psi : L \rightarrow G$ for some language 
	$L \subseteq \Sigma^*$.  
	Let $H \leqslant G$ be a finitely generated subgroup 
	of $G$ and $S' = A' \cup A'^{-1}= 
	\{s_1', \dots, s_k '\}$ be a 
	set of semigroup generators of $H$. Let 
	$L' = \psi^{-1}(H) \subset L$.  
	We define $\psi': L' \rightarrow H$ as the 
	restriction of $\psi$ onto $L'$: 
	for a given $w \in L'$, $\psi'(w) = \psi(w)$. 
	In order to prove that the representation 
	$\psi': L' \rightarrow H$ is \linearC\ we repeat the argument from  
	Proposition \ref{generator_independence}.  
	Let $f_i : \Sigma^* \rightarrow \Sigma^*$  
	be automatic functions corresponding 
	to multiplications in $G$ by the semigroup 
	generators $s_i$, for $i =1, \dots, n$  respectively: 
	$\psi (f_i (w)) = \psi (w) s_i$ for all $w \in L$.  
	For a given $j =1, \dots , k$ there exist 
	$j_1, \dots, j_m$ for which 
	$s_j ' = s_{j_1} \dots s_{j_m}$. 
	For every $j=1,\dots,k$ the function 
	$f_j' = f_{j_m} \circ \dots \circ f_{j_1}$ is automatic, 
	$f_j ' (L') \subseteq L'$ and 
	$\psi' (f_j '(w)) = \psi ' (w) s_j '$ for all $w \in L'$. 
	Therefore, the group $H$ is \linearC. 
\end{proof}
\begin{remark}	 
	We remark that the language $L'$ in the proof of 
	Theorem \ref{subgroupsofquasiCayleyarequasiCayley} 
	is not necessarily in the same class as the language 
	$L$. An illustrative example, when 
	$L$ is a regular language but 
	$L'$ is not recursive, is 
	shown in Proposition \ref{Mikhailova_example}.     
\end{remark}
\subsection{Relation with $\mathcal{C}$--graph Automatic Groups}\label{subsec:CGA}

%

Let  $\Sigma$ be a finite alphabet and the symbol 
$\diamond$  not in $\Sigma$. We define
$\Sigma_\diamond = \Sigma \cup \{\diamond\}$.      
For two given strings 
$u_1,u_2 \in \Sigma^*$, 
the convolution $u_1 \otimes u_2$   
is the string of length 
$\max\{|u_1|,|u_2|\}$
over the alphabet
$\Sigma_\diamond^2$ for which the $i$th symbol 
is $\sigma_i ^1 \choose \sigma_i ^2$, 
where $\sigma_i ^k$  is the $i$th symbols of
$u_k$ if $i \leqslant |u_k|$ and $\diamond$, otherwise, 
for $k =1,2$ and  
$i= 1,\dots, \max\{|u_1|,|u_2|\}$.

In order to extend the class of Cayley automatic groups, 
the second author and Taback introduced the notion of a 
$\left(\mathcal{B},\mathcal{C}\right)$--graph 
automatic group \cite{ElderTabackCgraph}.  
Let $G$ be a group, $S$ be a symmetric generating set of $G$ and $\Sigma$ be a finite alphabet. 
A tuple $(G,S,\Sigma)$ is said to be 
$(\mathcal{B},\mathcal{C})$--graph 
automatic if there is a bijection 
$\psi : L \rightarrow G$ between a language 
$L \subseteq \Sigma ^*$ from the class 
$\mathcal{B}$ and a group $G$ such that 
for every $s \in S$  the  language 
$L_s = \{u \otimes v \, | \, u,v \in L,
\psi(u)s = \psi(v)\}$ is in the class $\mathcal{C}$. 
If $\mathcal{B}=\mathcal{C}$, then the tuple 
$(G,S,\Sigma)$ is said to be $\mathcal{C}$--graph 
automatic. 
\begin{theorem}
	\label{CquasiCayley_are_Cgraphautomatic}
	Assume that a class of languages 
	$\mathcal{C}$ satisfies the following properties: 
	\begin{enumerate}[(a)] 
		\item{if $L \subseteq \Sigma^*$ is some language in 
			the class $\mathcal{C}$, then 
			$L \otimes \Sigma^*= 
			\{u \otimes v | u \in  L, v \in \Sigma^*\}$ is in the class $\mathcal{C}$;} 
		\item{if $R$ is a regular language and $L$ is in
			the class $\mathcal{C}$, then $R \cap L$ is 
			in the class $\mathcal{C}$.}   
	\end{enumerate}   
	Then, for a given 
	\ClinearC\ group $G$, the tuple $(G,S,\Sigma)$  is	$\mathcal{C}$--graph automatic for some alphabet 
	$\Sigma$ and every symmetric generating set $S$.  
\end{theorem}
\begin{proof} 	  
	Let $G$ be a \ClinearC\ 
	group for some class $\mathcal{C}$ satisfying the 
	conditions (a) and (b) of the theorem. 
	Then there exists a \ClinearC\ representation 
	$\psi : L \rightarrow G$ for some language 
	$L \subseteq \Sigma^*$ in the class $\mathcal{C}$. 
	By the condition (a), the language
	$L \otimes \Sigma^*$ is in the class  
	$\mathcal{C}$. Let $A$ be a set of generators 
	of $G$. For a given semigroup generator 
	$s \in S =  A \cup A^{-1}$ there exists an 
	automatic function 
	$f_s : \Sigma^* \rightarrow \Sigma^*$  
	such that $f_s (L) \subseteq L$ and 
	$\psi (f_s(w)) = \psi(w)s$ for all $w \in L$. 
	Since $f_s$ is automatic, the language  
	$R_s = \{u \otimes f_s (u) \, |\, u \in \Sigma^*  \} 
	\subseteq \Sigma^* \otimes \Sigma^*$ is regular. 
	Therefore, by the condition (b), the language 
	$ (L \otimes \Sigma^*) \cap R_s$ is in the class 
	$\mathcal{C}$. 
	Thus, for every $s \in S$ the  
	language $\{u \otimes v  \, | \, 
	u,v\in L, \psi (u) s = \psi (v) \} = 
	(L \otimes \Sigma^*) \cap R_s$  is 
	in the class $\mathcal{C}$, so 
	$(G,S,\Sigma)$ is $\mathcal{C}$--graph automatic. 
\end{proof}	
\begin{remark}
	We note that the condition imposed on the class $\mathcal{C}$ in Theorem \ref{CquasiCayley_are_Cgraphautomatic} 
	is satisfied for a wide family of languages 
	including all those mentioned in Remark 
	\ref{remarkconditions1}.   
\end{remark}

\subsection{Examples} 
 \label{examples_GL_nQ}
Thurston proved that an automatic 
nilpotent group must be virtually abelian 
\cite{Epsteinbook}. 
Kharlampovich, Khoussainov and 
Miasnikov showed that every f.g. nilpotent 
group of nilpotency class at most  two is 
Cayley automatic \cite{KKM11}. 
However, it is 
conjectured that 
there exists a f.g. nilpotent group of nilpotency 
class three which is not Cayley automatic 
\cite{SemiautomaticNilpGroups}. 
The main purpose of this subsection is to show     
that \linearC\ groups comprise 
a wide family of groups including all 
f.g. subgroups of $\mathrm{GL}(n,\mathbb{Q})$. 
This implies that all polycyclic 
groups are \linearC. 
The latter, in 
particular, shows that all f.g. nilpotent groups are \linearC. 
The groups  $\mathrm{SL}(n,\mathbb{Z})$ are 
also \linearC.\footnote{We recall that 
	$\mathrm{SL}(2,\mathbb{Z})$ is automatic; so, it 
	is also Cayley automatic. 
	It is not known whether the groups 
	$\mathrm{SL}(n,\mathbb{Z})$ for $n>2$ are 
	Cayley automatic or not.}     
\begin{theorem}
	\label{fgsubgroupglnq}	  
	A finitely generated subgroup of 
	$\mathrm{GL}(n,\mathbb{Q})$ is  \linearC.   
\end{theorem}	
\begin{proof} 
	Let $G$ be a f.g. subgroup of 
	$\mathrm{GL}(n,\mathbb{Q})$ and $S$ be a 
	set of semigroup generators of $G$. 
	Each $s \in S$ corresponds to a 
	matrix $M_s \in \mathrm{GL}(n,\mathbb{Q})$ 
	with rational coefficients 
	$m_{s,ij} = \frac{p_{s,ij}}{q_{s,ij}}$
	for $i,j =1,\dots,n$, 
	where $p_{s,ij}, q_{s,ij} \in \mathbb{Z}$ and
	$q_{s,ij}>0$.   
	Now we notice that there exist an integer $k>0$
	and integers $r_{s,ij}$ such that 
	$m_{s,ij} = \frac{r_{s,ij}}{k}$ for all 
	$s \in  S$ and $i,j =1,\dots, n$; 
	for example, one can put 
	$k = \prod\limits_{s \in S} 
	\prod \limits_{i,j =1}^{n} q_{s,ij}$. 
	Therefore, we may assume  
	that for all $s \in S$ and 
	$i,j =1,\dots,n$: 
	$m_{s,ij} \in \mathbb{Z}[1/k]$, where 
	$\mathbb{Z}\left[1/k\right]$ is the abelian group of 
	all rational numbers of the form 
	$\frac{d}{k^\ell}$ for $d, \ell \in \mathbb{Z}$ and 
	$\ell \geqslant 0$. 
	For example, if $k =10$, then 
	$\mathbb{Z}\left[1/k\right]$ is just 
	the group of all finite fractional 
	decimal numbers, i.e., the rational 
	numbers for which the number of digits 
	after the dot is finite.   
	Since all coefficients of the matrices 
	$M_s$, $s \in S$ are in 
	$\mathbb{Z}\left[1/k\right]$, 
	then for every matrix from $G$ 
	the coefficients of this matrix are 
	also in $\mathbb{Z}[1/k]$. That is, 
	$G$ consists of matrices with coefficients 
	from $\mathbb{Z}\left[1/k\right]$. Therefore, 
	$G \subset M_n (\mathbb{Z}\left[1/k\right])$, 
	where $M_n (\mathbb{Z}\left[1/k\right])$ is the
	ring of $n \times n$ matrices with coefficients in 
	$\mathbb{Z}\left[1/k\right]$.

	The abelian group 
	$\left(\mathbb{Z}\left[1/k\right],+\right)$ 
	is FA--presentable, 
	see the proof, e.g., in 
	\cite{NiesSemukhin07}. 
	If $k=10$,  
	then one can simply use the standard 
	decimal representation of numbers from   
	$\mathbb{Z}\left[1/k\right]$.
	For other values of $k$, one can use a 
	representation in base $k$.   
	Let us choose any FA--presentation of 
	$\left( \mathbb{Z} \left[ 1/k \right], +\right)$, i.e., 
	a bijection 
	$\varphi : L_1 \rightarrow \mathbb{Z}\left[1/k\right]$ 
	from some regular language $L_1$ to         
	$\mathbb{Z}\left[1/k\right]$ for which 
	the relation $R_{+} = \{ \left(u,v,w \right) 
	\in L_1 \times L_1 \times L_1 \,|\,\varphi(u) + 
	\varphi(v) = \varphi(w) \}$ is FA--recognizable.
	The latter also implies that multiplication 
	by any fixed number $t = \frac{p}{k^i} \in 
	\mathbb{Z}\left[1/k \right]$ is FA--recognizable.
	That is,  
	the relation 
	$R_t = \{\left(u, v\right)  \in L_1 \times L_1 \, | \, 
	\varphi (u) t  =\varphi(v)\}$
	is FA--recognizable.  
	Now, every matrix
	$C \in M_n(\mathbb{Z}\left[1/k\right])$ 
	with coefficients 
	$c_{ij} \in  \mathbb{Z}\left[1/k\right]$ 
	for $i,j=1,\dots,n$ we represent as 
	the convolution 
	$\varphi^{-1}(c_{11}) \otimes \varphi^{-1}(c_{12})
	\otimes  \dots \otimes \varphi^{-1}(c_{nn})$. 
	The collection of all such convolutions 
	form a regular language
	$L_n 
	= \{u_{11} \otimes \dots \otimes u_{nn}\, |\, 
	u_{ij} \in L_1, i,j=1,\dots,n\}$. 
	This  gives the  bijection 
	$\varphi_n : L_n \rightarrow M_n (\mathbb{Z}\left[1/k\right])$
	between $L_n$ and 
	$M_n (\mathbb{Z}\left[1/k\right])$. 
	
	For a given matrix $C$, 
	the result of 
	the multiplication of $C$ 
	by a matrix $M_s$ for $s \in S$ is 
	given by the following: 
	for given $i,j=1,\dots, n$, the 
	coefficient $d_{ij}$ of the matrix $D = C M_s$ 
	equals 
	$d_{ij} = c_{i1} m_{s,1j} + \dots + c_{in} 
	m_{s,nj}$. 
	Therefore, since $R_+$ and 
	$R_t$ for all $t \in \mathbb{Z}\left[1/k\right]$
	are FA--recognizable, the  
	relation 
	$R_s = \{ \left(u,v \right) \in L_n \times L_n \, |\, 
	\varphi_n(u) M_s = \varphi_n (v) \}$ is 
	FA--recognizable for every $s \in S$. 
	Let $L = \{ w \in L_n \, | \, \varphi_n (w ) \in G\}$ 
	and $\psi$ be the restriction of $\varphi_n$ onto 
	$L$. Then $\psi : L \rightarrow G$ is a 
	\linearC\ representation of $G$. 	 
\end{proof} 

\begin{corollary} 
	\label{virtpolquasiCayley_cor}  
	A virtually polycyclic group is 
	\linearC. 
\end{corollary}
\begin{proof}
	By Theorem \ref{finiteextensionthm}, 
	it is enough only to show that a polycyclic group is
	\linearC. 
	Auslander showed that a polycyclic group has 
	a faithful representation in 
	$\mathrm{SL}(n,\mathbb{Z})$ 
	\cite{Auslander67}. 
	Therefore, a polycyclic group is isomorphic 
	to a f.g. subgroup of $\mathrm{GL}(n,\mathbb{Q})$
	which is \linearC\   
	by Theorem \ref{fgsubgroupglnq}.       	 
\end{proof}

%
%

\section{Cayley Polynomial--Time Computable Groups} 
\label{Cayleypoltimecomp_section} 

The notion of a \linearC\ group 
can be extended further to that  of a 
\polyC\ group which we introduce in this section.

We say that a function 
$f : \Sigma^* \rightarrow \Sigma^*$ is 
polynomial--time computable if 
it is computed by a deterministic
one--tape Turing machine in time $O(p(n))$, 
where $p(n)$ is a polynomial and $n$ is a 
length of the input.  
Note that  (in contrast to the linear--time case) 
restricting to position--faithful Turing machines
has no effect: a function computed by a deterministic 
one--tape Turing machine in polynomial time can be  computed by a deterministic position--faithful one--tape Turing machine in polynomial time by performing the same steps and at the end copying the output to the front of the tape 
(this takes at most polynomial time).

Let $G$ be a f.g. group 
and $S= \{s_1, \dots, s_n \} \subseteq G$ be 
a finite set of semigroup generators of $G$. 
Let  $\mathcal{C}$ be a nonempty class of 
languages. 
\begin{definition}[Cayley polynomial--time computable groups] 
	\label{definition_Cayley_polynomial}	
	We say that the group $G$ is 
	$\mathcal{C}$--Cayley polynomial--time  
	computable if there exist a language  
	$L \subseteq \Sigma^*$ from the class 
	$\mathcal{C}$ over some finite alphabet 
	$\Sigma$, a bijective mapping 
	$\psi : L \rightarrow  G$ between the 
	language $L$ and the group $G$ 
	and polynomial--time computable  
	functions 
	$f_i : \Sigma ^* \rightarrow \Sigma^*$ 
	such that $f_i (L) \subseteq L$ and for 
	every $w \in L$: 
	$   \psi (f_i (w)) = \psi(w) s_i$,    
	for all $i=1,\dots,n$. We call 
	$\psi : L \rightarrow G$ a 
	$\mathcal{C}$--Cayley polynomial--time 
	computable representation of the group 
	$G$. If the requirement for $L$ to be 
	in a specific class $\mathcal{C}$ is omitted, 
	then we just say that $G$ is a Cayley 
	polynomial--time computable group and
	$\psi: L \rightarrow G$ is a Cayley 
	polynomial--time computable representation 
	of $G$.  
\end{definition}  
A \ClinearC\ group is \CpolyC.   
Similarly to \ClinearC\ 
groups, the notion of a $\mathcal{C}$--Cayley 
polynomial--time computable group does not depend
on the choice of generators.  
\begin{proposition} 
	\label{ind_choice_gen_caylepol}	 
	The notion of a \CpolyC\ 
	group does not depend on the choice 
	of generators. 	
\end{proposition}	
\begin{proof} 
	We first notice that if the given functions 
	$f_{j_i}: \Sigma^* \rightarrow \Sigma^*, i=1,\dots,m$ 
	are polynomial--time computable, then the 
	composition $f_{j_m} \circ \dots \circ f_{j_1}$ 
	is polynomial--time computable.
	The rest literally repeats 
	the proof of Proposition \ref{generator_independence} 
	modulo changing the term automatic 
	(p.f. linear--time) 
    to polynomial--time.  
\end{proof}	
\begin{remark}
	We note that if the degree of a polynomial 
	$p(n)$ is greater than one, then the 
	composition $g \circ f$ of two functions 
	$f$ and $g$ computed in time $O(p(n))$ is 
	in general not necessarily computed in 
	time $O(p(n))$; we may only guarantee it is 
	computed in time $O(p(p(n)))$. So, 
	fixing an upper bound for the time complexity 
	in Definition \ref{definition_Cayley_polynomial}
	one  loses the independence 
	on the choice generators. An alternative 
	approach could be to use a more powerful 
	computational model (for example, 
	a two--tape Turing machine) and 
	force the time complexity to be at most linear.  
	In this case one gets  
	the independence on the choice generators without
	needing to update the time complexity.   
\end{remark}	 

 We say that $G$ is  
 \CpolyC\ with quasigeodesic normal form if 
 there is a \CpolyC\ representation 
 $\psi : L \rightarrow G$ which has quasigeodesic
 normal form, see Definition \ref{quasigeodesicnf_def}. 
 If the requirement for $L$ to be in a specific 
 class $\mathcal{C}$ is omitted, then we simply say
 that $G$ is \polyC\ with quasigeodesic 
 normal form.
Note that a \polyC\
representation $\psi: L \rightarrow G$ does not 
necessary have quasigeodesic normal form 
like every \linearC\ representation, see Theorem \ref{quasigeodesicnormalformthm}. 
Therefore, the argument used in the proof of Theorem 
\ref{quadratic_alg_theorem1} cannot be generalized 
for an arbitrary Cayley--polynomial time computable representation. However, the following  
analogue of Theorem \ref{quadratic_alg_theorem1} 
holds: 
\begin{theorem}[Computing normal form in 
	polynomial time]
	\label{normalform_poltime}	              
	Suppose that a Cayley polynomial--time 
	computable representation $\psi : L \rightarrow G$ 
	has quasigeodesic normal form. 
	Then there is an algorithm 
	which for a given input word 
	$v = s_1\dots s_k \in (A \cup A^{-1})^*$ 
	computes the string $u \in L$ for which 
	$\psi(u) = \pi (v)$. Moreover, this algorithm 
	can be implemented by a deterministic 
	one--tape Turing machine 
	in polynomial time.   
\end{theorem}	
\begin{proof} 
	The proof repeats Theorem \ref{quadratic_alg_theorem1}
	modulo the following minor changes. 
	Since $\psi: L \rightarrow G$ has
	quasigeodesic normal form, for every string 
	$u_{j-1} = \psi^{-1}(s_1 \dots s_{j-1})$,  
	$j=1,\dots,k$, the following inequality is satisfied: 
	$|u_{j-1}| \leqslant  C(d_A (s_1 \dots s_{j-1}) +1)
	\leqslant C (j-1) + C \leqslant C k + C$ 
	for some constant $C$. 
	Therefore, polynomial time is required to compute 
	the string $u_{j}$ from $u_{j-1}$. So the total 
	time required to compute the string $u_k$ from 
	the input $s_1\dots s_k$ is polynomial.         
\end{proof}	 
Similarly to Corollary \ref{quasiCayleywordproblem} 
we immediately obtain the following. 
\begin{corollary}[Solving word problem in polynomial time]
	\label{wordproblempoltime_cor}   
	If a given group $G$ is \polyC\ with 
    quasigeodesic normal form, the 
	word problem in $G$ can be solved by a deterministic 
	one--tape Turing machine in polynomial time.  	
\end{corollary}	
Clearly, the analogue of Theorem \ref{L_recursivelyenumerable_quasiautomatic} holds 
for a Cayley polynomial--time computable 
representation 
$\psi : L \rightarrow G$. 
\begin{theorem} 
	\label{caypoltime_normforms_recenum}	
	For every Cayley polynomial--time computable 
	representation $\psi : L \rightarrow G$ the 
	language $L$ is in the class \RE.   
\end{theorem}

Furthermore, all closure properties
with respect to taking a finite extension, 
the direct product, the free product and a finitely 
generated subgroup, shown in Theorems \ref{finiteextensionthm}, 
\ref{directproducttheorem}, \ref{freeproducttheorem} 
and \ref{subgroupsofquasiCayleyarequasiCayley}, 
respectively, remain valid 
for Cayley polynomial--time computable groups and 
the ones with quasigeodesic normal forms. Namely, we have the 
following.
\begin{theorem}[Finite extensions, direct products, 
	free products] 
	\label{Finite_extension_DP_FP_CPTC_thm}	 
	For a given class of languages $\mathcal{C}$, 
	assuming that the relevant conditions are 
	satisfied \footnote{See the conditions on the class
		$\mathcal{C}$ in Theorems \ref{finiteextensionthm}, \ref{directproducttheorem} 
		and \ref{freeproducttheorem} for finite extensions,
		direct products and free products, respectively.}, 
	the class of $\mathcal{C}$--Cayley polynomial--time 
	computable groups is closed under 
	taking a finite extension, the direct 
	product and the free product. The same 
	holds for the class of $\mathcal{C}$--Cayley polynomial--time computable groups 
	with quasigeodesic normal forms.  
\end{theorem}
\begin{proof} 
   For the first statement of the theorem 
   the proof repeats Theorems \ref{finiteextensionthm}, 
   \ref{directproducttheorem} and \ref{freeproducttheorem}
   with minor obvious changes.
   Note that, though we use 
   Lemma \ref{bounded_difference_lemma} 
   in the proof of Theorem \ref{directproducttheorem}, 
   we do not need it for proving the analogous result  
   for $\mathcal{C}$--Cayley polynomial--time computable groups as shifting of a string on the tape
   by at most polynomial number 
   of cells requires at most polynomial time.  
   For the second statement 
   of the theorem it is enough to notice that 
   the quasigeodesic property is preserved 
   for all representations which appear in the proofs of these theorems.      
\end{proof}	
  
\begin{theorem}[Finitely generated subgroups]
	\label{fgsubgroups_caypoltime}	
	The class of Cayley 
	polynomial--time computable 
	groups is closed under taking a finitely generated 
	subgroup. The same holds for the class of 
	Cayley polynomial--time computable groups  
	with quasigeodesic normal forms. 
\end{theorem}
\begin{proof}
   For the first statement of the theorem 
   the proof repeats Theorem \ref{subgroupsofquasiCayleyarequasiCayley}
   with obvious changes. 
   In order to show the second statement of the theorem, 
   in the proof of 
   Theorem \ref{subgroupsofquasiCayleyarequasiCayley} 
   it is enough to notice that if 
   $\psi: L \rightarrow G$ has 
   quasigeodesic normal form, 
   then for each $w \in L' \subset L$ the inequalities
   $|w| \leqslant C (d_A (\psi(w)) + 1) 
   \leqslant C (C' d_{A'}(\psi(w)) +1)$ hold
   for some constants $C,C'>0$. This implies that 
   $\psi': L' \rightarrow H$, which is the 
   restriction of $\psi$ onto $L'$, 
   also has quasigeodesic normal form.
\end{proof}

	We note that Theorem 
	\ref{CquasiCayley_are_Cgraphautomatic} 
	cannot be directly generalized to  
	$\mathcal{C}$--Cayley polynomial--time 
	computable groups, so in 
	general we cannot say how they compare with  
	$\mathcal{C}$--graph automatic 
	groups. 
	However, for some special classes of languages 
	$\mathcal{C}$, it is possible to relate 
	 these two classes  
	of groups
	as we explain in the following remark.
	
	\begin{remark}
		\label{Sk-graph_automatic_remark}
	In \cite[Theorem~10]{ElderTabackCgraph} the second 
	author and Taback showed that 
	for a $\mathscr{S}_k$--graph automatic group with 
    quasigeodesic normal form, 
	the normal form is computable in polynomial time, 
	where $\mathscr{S}_k$ is the class of languages 
	accepted by a non--blind non--deterministic 
	$k$--counter automaton running in quasi--realtime\footnote{We recall that 
	a non--blind non--deterministic 
	$k$--counter automaton
	 is 
	a non--deterministic automaton augmented with 
	$k$ integer counters which are initially set to zero 
	\cite{BookGinsburg72}. 
	These counters can be incremented, decremented, set to 
	zero and compared to zero. Running in quasi--realtime 
	means that the number of allowed  
	consecutive $\epsilon$--transitions is bounded from
	above by some constant.
	A string is accepted by this automaton exactly if it reaches an accepting state with all counters returned to zero.}.      
	Now let $G$ be a $\mathscr{S}_k$--graph automatic 
	group with a symmetric set of semigroup generators 
	$S = \{s_1, \dots, s_n \} \subset G$. 
	Then there is a bijection 
	$\psi : L \rightarrow G$ from a language 
	$L \subseteq \Sigma^*$ in the class $\mathscr{S}_k$ 
	to $G$ for which 
	the languages $L_s = \{u \otimes v \,|\, 
	u,v \in L,  \psi (v) = \psi(u) s \}$  
	are in $\mathscr{S}_k$ for every $s \in S$. 
	Assume that there is a polynomial 
	$p$ such that for all $s \in S$ and $u,v \in L$,
	for which $\psi(v) = \psi(u)s$, the 
	following inequality holds:  
	\begin{equation}
	\label{remark_Sk_inequality}  
	  |v| \leqslant p (|u|).
	\end{equation} 
	Then $G$ is $\mathscr{S}_k$--Cayley polynomial--time computable and $\psi : L \rightarrow G$  
	is a $\mathscr{S}_k$--Cayley polynomial--time computable 
	representation of $G$. 
	In order to prove this
	one only needs to show that for a given 
	$s \in S$ there is a polynomial--time 
	algorithm which for the input $u \in L$
	produces the output $v \in L$ such that  
	$u \otimes v \in L_s$. 
	The reader may look up this algorithm 
	and the explanation why it runs in 
	polynomial time in \cite[Theorem~10]{ElderTabackCgraph}.
     We note that 
    \eqref{remark_Sk_inequality} necessarily holds if
    $\psi : L \rightarrow G$ is a Cayley polynomial--time 
    computable representation.
\end{remark}

What are examples of \polyC\ groups? Especially we 
are interested in examples of 
$\REG$--\polyC\ groups because,
similarly to \linearC\ groups, 
this class naturally extends the class of 
Cayley automatic groups.

First, in order to show that the class of
$\REG$--\polyC\ groups is wide,
we show that it comprises all f.g. nilpotent groups.    
Let $G$ be a f.g. nilpotent group. 
Suppose first that $G$ is torsion--free.  
There is a central series 
$G = G_1 \geqslant  
\dots \geqslant G_{n+1} = 1$ such that  
$G_i / G_{i+1}$ is an infinite cyclic group for 
all $i =1,\dots,n$. Then there exist  
$a_1,\dots,a_n \in G$ for which
$G_i = \langle a_i , G_{i+1} \rangle$. 
This implies that every element $g \in G$ 
has a unique normal form 
$g = a_1 ^{x_1} \dots a_n ^{x_n}$, where  
$x_1,\dots,x_n$  are integers. 
Let $L$ be a language of such normal forms 
over the alphabet 
$\Sigma = \{a_1,\dots,a_n,a_1 ^{-1},\dots,a_n ^{-1}\}$. 
Clearly, $L$ is a regular language. 
The canonical mapping $\pi : L \rightarrow G$ 
gives a bijection between $L$ and $G$.   
For given two group elements  
$g_1 = a_1 ^{x_1} \dots a_n ^{x_n}$ 
and $g_2= a_1 ^{y_1} \dots a_n ^{y_n}$, the product  
$g_1 g_2$ equals $a_1 ^{q_1} \dots a_n ^{q_n}$ 
for some integers $q_1, \dots q_n$. This defines the 
functions $q_i (x_1,\dots,x_n, y_1,\dots,y_n)$, 
$i=1,\dots,n$ of $2n$ integer variables
$x_1,\dots,x_n,y_1, \dots, y_n$; in fact 
it can be shown that $q_i$ depends only on $x_1,\dots,x_i$ and $y_1,\dots,y_i$ for every $i=1,\dots,n$. 
Hall showed \cite{Hall57} that the functions 
$q_i$ are polynomials 
$q_i \in \mathbb{Q}[x_1,\dots, x_n, y_1,\dots, y_n]$.
Therefore, for each semigroup generator 
$s \in \{a_1,\dots, a_n, a_1^{-1},\dots, a_n^{-1} \}$ 
there exist polynomials 
$p_{s,i} \in \mathbb{Q}[x_1,\dots,x_n]$ 
for $i=1,\dots,n$ such that 
$a_1 ^{x_1} \dots a_n ^{x_n} s = 
a_1 ^{p_{s,1}}\dots a_n ^{p_{s,n}}$.  
It can be seen that these right multiplications are 
polynomial--time computable functions.  
Therefore, $\pi : L  \rightarrow G$
is a $\REG$--Cayley polynomial--time computable 
representation.
If the group $G$ is not torsion--free, it  
has a torsion--free nilpotent subgroup of 
finite index. Therefore, by Theorem \ref{Finite_extension_DP_FP_CPTC_thm}, 
$G$ also has a $\REG$--Cayley polynomial--time computable 
representation. 


Other nontrivial examples of 
$\REG$--\polyC\ groups include 
the wreath product $\mathbb{Z}_2 \wr \mathbb{Z}^2$ 
and Thompson's group $F$. 
We denote by $\mathsf{IND}$ the class of indexed 
languages\footnote{We recall that indexed languages
are languages recognized by 
\href{https://en.wikipedia.org/wiki/Nested_stack_automaton}
{nested stack automata}.}. 
In \cite[\S~5]{BK15}  the first author and 
Khoussainov showed that 
the group $\mathbb{Z}_2 \wr \mathbb{Z}^2$ is 
$(\REG,\mathsf{IND})$--graph 
automatic\footnote{
	$(\mathcal{B,C})$--graph automatic is defined in \cite{ElderTabackCgraph}: the normal form is in the language class $\mathcal B$ and the 2-tape language for multiplication is in the class $\mathcal C$. }
by constructing  
a certain bijection between a regular language and 
the group $\mathbb{Z}_2 \wr \mathbb{Z}^2$.  
It can be verified that this bijection  
is a \REG--Cayley polynomial--time computable 
representation of  $\mathbb{Z}_2 \wr \mathbb{Z}^2$. 
Therefore,  $\mathbb{Z}_2 \wr \mathbb{Z}^2$ 
is a \REG--Cayley polynomial--time computable 
group. This representation does not have 
quasigeodesic normal form, see \cite[Remark~9]{BK15}.

Let $\mathsf{D}\mathscr{S}_1$ be the class of 
non--blind deterministic $1$--counter languages, 
see the footnote in  Remark \ref{Sk-graph_automatic_remark} where the definition of non--blind 
non--deterministic $k$--counter languages is recalled. 
In \cite{ElderTaback15_Thompson} the second author 
and Taback showed that Thompson's group $F$ is 
$\left(\REG,\mathsf{D}\mathscr{S}_1 \right)$--graph 
automatic. Moreover, for their representation  the inequalities 
of the form 
$\lambda |w| - \lambda_0 \leqslant l (g) 
 \leqslant \mu |w| + \mu_0$ hold for all $g \in F$, 
 where $l(g)$ is the length of a geodesic word representing 
 $g$, $w$ is the normal form corresponding to $g$ and 
 $\lambda,\mu, \lambda_0, \mu_0 >0$, 
 see \cite[Proposition~3.3]{ElderTaback15_Thompson}. 
 These inequalities imply that an inequality 
 of the form \eqref{remark_Sk_inequality} holds 
 for some linear function $p$. 
 By the observation made in the end of 
 Remark \ref{Sk-graph_automatic_remark},
 we obtain that $F$ is $\REG$--Cayley polynomial--time 
 computable. 
 Moreover, the inequality  
 $\lambda |w| - \lambda_0 \leqslant l (g)$ implies that 
 $F$ is $\REG$--Cayley polynomial--time 
 computable with quasigeodesic normal form.


For the last example of a 
\polyC\ group we mention the wreath 
product $\mathbb{Z}_2 \wr \mathbb{F}_2$. 
We denote by $\DCFL$ the class of deterministic context--free languages.
In \cite[\S~4]{BK15}
it was shown that the group $\mathbb{Z}_2 \wr \mathbb{F}_2$ 
is $\DCFL$--graph automatic
by constructing a certain bijection between 
a $\DCFL$
language and the group $\mathbb{Z}_2 \wr \mathbb{F}_2$. It can be verified 
that this bijection is a $\DCFL$--Cayley 
polynomial--time computable representation of 
$\mathbb{Z}_2 \wr \mathbb{F}_2$. 
Moreover, for this representation the 
inequalities 
of the form
$\lambda |w| - \lambda_0 \leqslant l (g) 
\leqslant \mu |w| + \mu_0$ hold for all 
$g \in \mathbb{Z}_2 \wr \mathbb{F}_2$,
see \cite[Theorem~5]{BK15}.  
The inequality $\lambda |w| - \lambda_0 \leqslant l (g)$ 
implies that $\mathbb{Z}_2 \wr \mathbb{F}_2$ is 
$\DCFL$--Cayley polynomial--time 
computable with quasigeodesic 
normal form.



\section{\distfun\ for \polyC\ groups}
\label{characterization_section}

 Let $G$ be a f.g. group with 
 a finite generating set $A \subset G$.  
 Let $\psi: L \rightarrow G$ be a bijection
from a language $L \subseteq \Sigma^*$ 
to the group $G$.
For each symbol $\sigma \in \Sigma$ 
one can assign a group element $g_\sigma \in G$. 
This assignment defines a mapping 
$\alpha : \Sigma \rightarrow G$, 
not necessarily injective, for which 
$\alpha (\sigma ) = g_\sigma$ for all 
$\sigma \in \Sigma$. 
Then we can define the canonical mapping 
$\pi_\alpha : L \rightarrow G$ as follows: 
for a given string
$w  = \sigma_1 \dots \sigma_k \in L$ we define 
$\pi_\alpha(w) \in G$ as $\pi_\alpha (w) = \alpha(\sigma_1) \alpha(\sigma_2) \dots \alpha(\sigma_k)$
and $\pi_\alpha(w)=e$ if $w= \epsilon$. Thus, for 
fixed $\psi: L \rightarrow G$ and 
$\alpha: \Sigma \rightarrow G$, the following 
nondecreasing function 
$h_{\psi,\alpha}: \left[N, + \infty \right) 
\rightarrow \mathbb{R}^+$ is defined by:  
\begin{equation} 
\label{h_function_def}   
h_{\psi,\alpha}(n) = \max\{d_A (\pi_\alpha(w),\psi(w))\,|\,
w \in L^{\leqslant n} \},  
\end{equation} 
where $d_A (\pi_\alpha(w),\psi(w))$ is the distance 
between $\pi_\alpha (w)$ and $\psi(w)$ in the word
metric relative to $A$, which is the length of a 
geodesic word representing 
$\pi_\alpha(w)^{-1} \psi(w)$, i.e.,  
$d_A (\pi_\alpha(w),\psi(w))=  
d_A (\pi_\alpha(w)^{-1} \psi(w))$, 
$L^{\leqslant n} = \{w \in L \, | \, 
|w| \leqslant n\}$ and 
$N = \min \{n \in \mathbb{N} \, | \,
L^{\leqslant n} \neq \varnothing\}$.  
For given $\psi$ and $\alpha$ we call 
$h_{\psi,\alpha}$ a {\em \distfun}.     
This function 
was introduced  in 
\cite{BT_LATA18} and studied in \cite{eastwest19,BET19} in the context of Cayley automatic groups\footnote{In \cite{BT_LATA18,eastwest19} 
	it is assumed from 
	the beginning that $L$ is a language over some 
	symmetric set of generators. 
	It can be seen that this assumption is purely 
	a matter of convenience and it does not 
	have any effect on the study of 
	the \distfun\ $h_{\psi,\alpha}$.}.

 Clearly, if $G$ is automatic, 
 then for an automatic representation 
 $\pi: L \to G$, $L \subset (A \cup A^{-1})^*$ 
 and a natural mapping $\alpha: A \cup A^{-1} 
 \rightarrow G$ for which $\alpha (s)=s$, 
 $s \in A \cup A^{-1}$  
 all values of the Cayley 
 distance function $h_{\pi,\alpha}$
 are equal to zero. 
 \cite[Theorem~8]{BT_LATA18} shows that
 if a group $G$ has  some Cayley automatic representation 
 $\psi : L \rightarrow G$ and mapping  
 $\alpha : \Sigma \rightarrow G$ for which the 
 \distfun\ $h_{\psi,\alpha}$ is bounded from above by a constant, 
 then $G$ must be automatic.
In \cite{BET19} the first two authors and Taback ask: can the \distfun\ become arbitrarily close to a constant function for some non-automatic Cayley automatic group?
Here we show that the answer is no when we generalise to \linearC\ and $\REG$--\polyC\ representations: 
we furnish examples which have  \linearC\ and $\REG$--\polyC\  representations 
for which the \distfun\ is zero.

 For given two nondecreasing functions 
 $h_1 : \left[N_1,+ \infty \right) \rightarrow \mathbb{R}^{+}$ 
 and $h_2 : \left[N_2, +\infty \right) \rightarrow \mathbb{R}^{+}$
 we say that $h_1 \preceq h_2$, if
 there exist integer constants 
 $K,M>0$ and 
 $N \geqslant \max \{N_1,N_2 \}$ such that 
 $h_1 (n) \leqslant K h_2 (Mn)$ for all 
 $n \geqslant N$. It is said that  
 $h_1 \asymp h_2$ if $h_1 \preceq h_2$ and 
 $h_2 \preceq h_1$. 
 We say that a Cayley automatic group $G$ is separated 
 from automatic groups if there exists a 
 non--decreasing unbounded function $f$ such that 
 $f \preceq h_{\psi,\alpha}$ for all Cayley automatic 
 representations $\psi : L \rightarrow G$ and 
 mappings $\alpha: \Sigma \rightarrow G$.  
 Let $\bold{z} : \mathbb{N} 
 \rightarrow \mathbb{R}^{+}$ be the zero function: 
 $\bold{z} (n) = 0$ for all $n \in \mathbb{N}$. 
 We say that a \distfun\ $h_{\psi,\alpha}$ 
 vanishes if $h_{\psi,\alpha} \asymp \bold{z}$: 
 this equivalently means that 
 $h_{\psi,\alpha}(n) = 0$ for 
 all $n \in \mathrm{dom} \,h_{\psi,\alpha}$.  
 We denote by $\mathfrak{i}: \mathbb{N} \rightarrow 
 \mathbb{R}^+$ the identity function: 
 $\mathfrak{i}(n) = n$ for all $n \in \mathbb{N}$. 
 \begin{theorem} 
 \label{lamplighter_theorem}	
 There exists a Cayley automatic 
 group $G$ separated from automatic groups but  
 for which the \distfun\ $h_{\psi,\alpha}$ 
 vanishes
 for some 
 \linearC\ representation $\psi$ 
 of $G$ and mapping $\alpha$.      
 \end{theorem}
 \begin{proof} 
 	 In order to prove the theorem one needs to 
 	 provide an example of a group $G$ 
 	 satisfying the condition of the theorem. 
 	 For such an example we take the 
 	 lamplighter group $G=\mathbb{Z}_2 \wr \mathbb{Z}$.  
 	 The lamplighter group 
 	 $\mathbb{Z}_2 \wr \mathbb{Z}$ is Cayley automatic 
 	 \cite{KKM11}; but not automatic 
 	 because it is not finitely presented \cite{Epsteinbook}. 
 	 By \cite[Theorem~13]{BT_LATA18}, 
 	 for a Cayley automatic representation 
 	 $\psi : L \rightarrow \mathbb{Z}_2 \wr 
 	 \mathbb{Z}$ and a mapping 
 	 $\alpha : \Sigma \rightarrow 
 	 \mathbb{Z}_2 \wr \mathbb{Z}$ the 
 	 corresponding function $h_{\psi,\alpha}$ given by 
 	 \eqref{h_function_def} is coarsely greater 
 	 or equal than $\mathfrak{i}$: 
 	 $\mathfrak{i} \preceq h_{\psi,\alpha}$. 
 	 Below we will show that there exist 
 	 a \linearC\ representation of 
 	 $\mathbb{Z}_2 \wr \mathbb{Z}$ and 
 	 mapping for which the \distfun\ vanishes. 
 	 
 	 Each element of the lamplighter group 
 	 $\mathbb{Z}_2 \wr \mathbb{Z}$ 
 	 is identified with a pair $(f,z)$, where 
 	 $f$ is a function 
 	 $f: \mathbb{Z} \rightarrow \{0,1\}$ with 
 	 finite support (that is, only for finitely many 
 	 integers $i$, $f(i) = 1$) and 
 	 $z$ is an integer. 
 	 We denote by $a$ the pair $(f_0,1)$, 
 	 where $f_0 (j) = 0$ for all 
 	 $j \in \mathbb{Z}$, and by $b$ the 
 	 pair $(f_1,0)$, where $f_1(j) = 0$ for 
 	 all $j \neq 0$ and $f_1(0)=1$.   
 	 The group elements $a$ and $b$ 
 	 generate  
 	 $\mathbb{Z}_2 \wr \mathbb{Z}$ and the right 
 	 multiplications by $a,a^{-1}$ and $b=b^{-1}$ 
 	 are as follows.
 	 For a given $g = (f,z) \in \mathbb{Z}_2 \wr \mathbb{Z}$, 
 	 $ga = (f, z+1)$, $ga^{-1} = (f,z-1)$ and  
 	 $g b = gb^{-1} =  (f',z)$, where $f' (i) = f(i)$ for 
 	 $i \neq z$, $f'(z) = 0 $ if 
 	 $f(z) = 1$ and $f'(z) = 1 $ if  
 	 $f(z) = 0$. 
 	 The identity $e$ of $\mathbb{Z}_2 \wr \mathbb{Z}$ 
 	 corresponds to the pair $(f_0,0)$.     
    \begin{algorithm}
	\caption{An algorithm for computing the 
		substring $u$ of $w$} 	 
	\label{alg_comput_u}
	\begin{algorithmic}[1]	
		\Procedure{Substring}{$f,z,\ell,r$}
		\State $i \leftarrow \ell$; 
		$u \leftarrow \epsilon$  	
		\While{$i \leqslant r $}  
		\If {$i>\ell$} $u \leftarrow ua$
		\EndIf
		\If{$f(i) = 1$} {$u \leftarrow u b$}
		\EndIf
		\If {$i = z$} {$u \leftarrow u\uparrow$}
		\EndIf 		
		\State $i \leftarrow i + 1$
		\EndWhile
		\State \Return $u$
		\EndProcedure	
	\end{algorithmic}
\end{algorithm}

 	 Let $g = (f,z)$ be a given group 
 	 element of $\mathbb{Z}_2 \wr \mathbb{Z}$. 
 	 Let $M = \max \{i \, | \, f(i)=1, i \in \mathbb{Z}\}$  
 	 and $m = \min \{i \, | \, f(i) = 1, 
 	 i \in \mathbb{Z} \}$.
 	 We define $r = \max \{z,M\}$  
 	 and $\ell = \min \{z, m\}$.   
 	 Let $\Sigma  = \{b,a,a^{-1},\uparrow,\#\}$. 
 	 We define a normal form $w \in \Sigma^*$ of 
 	 $g$ to be the string $w = a^{\ell} \# u \# a^{z-r}$, 
 	 where the substring $u$ is 
 	 computed by Algorithm \ref{alg_comput_u} for a given pair $(f,z)$.   
 	 Informally speaking, the normal form 
 	 $w = a^{\ell} \# u \# a^{z-r}$ 
 	 is obtained as follows. First the pointer
 	 moves from the position $i=0$ to the position $i=\ell$.
 	 After that the pointer moves to the right 
 	 scanning the values $f(i)$ 
 	 and the position of the lamplighter 
 	 $z$ until it reaches the position $i = r$. 
 	 Then the pointer moves to the left 
 	 until it reaches the position of the lamplighter $i=z$.  
 	 Let us give two examples. Let $g_1 = (f,z)$ be the pair 
 	 for which $z=1$, $f(-1)=1$, $f(0)=1$, $f(2) = 1$
 	 and $f(i)=0$ for all $i \neq -1,0,2$. The normal 
 	 form of $g_1$ is $a^{-1}\#baba\uparrow ab\#a^{-1}$.  
 	 Let $g_2 = (f,z)$ be the pair 
 	 for which $z=1$, $f(-2)=1$
 	 and $f(i)=0$ for all $i \neq -2$. The normal 
 	 form of $g_2$ is $a^{-1}a^{-1}\#baaa\uparrow\#$.
 	 	 
 	 Let $L$ be the language of all such normal forms. 
 	 We denote by $\mathscr{C}_1$ the class of
 	 languages recognized by 
 	 a (quasi--realtime) blind deterministic $1$--counter 
 	 automaton\footnote{We recall that a blind deterministic 
 	 	$1$--counter automaton is a finite automaton 
 	 	augmented by an integer counter, initially 
 	 	set to zero, which can be incremented and decremented, 
 	 	but not read. A string is accepted by such an automaton
 	 	exactly if it reaches an accepting state with the counter
 	 	returned to zero.}.
 	 It follows from the simple argument below that 
 	 the language $L$ is in the class $\mathscr{C}_1$.    
 	 Let $w = a^{\ell} \# u \# a^{z-r}$.  
 	 The substring $u$ is of the form $u = p \uparrow s$, where
 	 $p$ is the prefix of $u$ preceding the symbol 
 	 $\uparrow$ and $s$ is the suffix of $u$ following 
 	 the symbol $\uparrow$. 
 	 The counter is increased by one each time the automaton 
 	 reads the symbol $a$ in the suffix $s$. The counter 
 	 is decreased by one each time the automaton reads 
 	 the symbol $a^{-1}$ in the suffix $a^{z-r}$ of $w$ following the second symbol $\#$. 
 	 Then $w \in L$ if and only if the counter returns to 
 	 $0$.       
 	 
     
     Construction of automatic functions recognizing the right multiplications by $b$ and $a,a^{-1}$ 
     is easy. The right multiplications by 
     $a,a^{-1}$ 
     require verification of 
     cases when $\ell$ or $r$ change. 
     Namely, for the right multiplication by $a$, $\ell$ is increased by $1$ if  
     $z = \ell$ and $f(\ell)=0$ and 
     $r$ is increased by $1$ if 
     $z=r$. For the right multiplication by 
     $a^{-1}$, $\ell$ is decreased by $1$ 
     if $z= \ell$ and $r$ is decreased by $1$ 
     if $z=r$ and $f(r)=0$.  
     All these cases can be verified by 
     a finite automaton.

 	 Thus, we constructed a 
 	 $\mathscr{C}_1$--\linearC\ representation 
 	 $\psi : L \rightarrow \mathbb{Z}_2 \wr \mathbb{Z}$ 
 	 of the lamplighter 
 	 group $\mathbb{Z}_2 \wr \mathbb{Z}$ which sends
 	 a normal form $w = a^{\ell} \# u \# a^{z-r}$
 	 to the corresponding group element 
 	 $g=(f,z)$. 
 	 Now let $\alpha: \Sigma \rightarrow \mathbb{Z}_2 \wr \mathbb{Z}$ be the following 
 	 mapping: $\alpha (a)= a$, $\alpha (a^{-1})= a^{-1}$, 
 	 $\alpha (b) = b$, $\alpha (\uparrow) = e$ and
 	 $\alpha (\#) = e$. Clearly, for the 
 	 $\mathscr{C}_1$--\linearC\ representation $\psi$ and 
 	 the mapping $\alpha$, the \distfun\
 	 $h_{\psi,\alpha}$ 
 	 vanishes.    
 \end{proof}	

 \begin{theorem} 
 \label{baumslagthm}	
 	There exist a Cayley automatic group $G$ separated 
 	from automatic groups but for which the Cayley 
 	distance function $h_{\psi,\alpha}$ vanishes 
 	for some $\REG$--\polyC\ representation $\psi$ 
 	of $G$ and mapping $\alpha$. 
 \end{theorem}
 \begin{proof}     
 	Let us consider the Baumslag--Solitar 
 	groups $BS(p,q) =$ 
 	$\langle a,t \,|$ $t a^p t^{-1}  = a^q \rangle$
 	for $1 \leqslant p < q$. 
 	These groups are not automatic \cite{Epsteinbook}, but
 	they are Cayley automatic  
 	\cite{KKM11,dlt14}. 
 	By \cite[Corollary~2.4]{eastwest19}, for a 
 	Cayley automatic representation 
 	$\psi : L \rightarrow BS(p,q)$ and 
 	a mapping $\alpha: \Sigma  \rightarrow BS(p,q)$ 
 	the corresponding function $h_{\psi,\alpha}$ 
 	given by \eqref{h_function_def} 
 	is coarsely greater or equal than $\mathfrak{i}$: 
 	$\mathfrak{i} \preceq h_{\psi,\alpha}$. 
 	We will show that for the Baumslag--Solitar group   
 	$BS(p,q)$ there are  
 	a $\REG$--Cayley polynomial--time 
 	computable representation 
 	and a mapping 
 	for which the \distfun\  
 	vanishes.   
 	
 	As a HNN extension of the infinite cyclic 
 	group the Baumslag--Solitar group 
 	$BS(p,q)$ 
 	admits the following normal form, see, e.g., 
 	\cite[Chapter~IV]{LyndonSchuppbook}.       
 	Every group element $g \in BS(p,q)$ can 
 	be uniquely written as a  
 	freely reduced word over the alphabet  
 	$\Sigma = \{a, a^{-1}, t, t^{-1}\}$
 	of the form
 	$
 	w_\ell t^{\varepsilon_\ell} \dots
 	w_1 t^{\varepsilon_1} a^k,    
 	$
 	where $\varepsilon_i \in \{+1,-1\}$, $k \in \mathbb{Z}$,  
 	$w_i = \{\epsilon, a, \dots, a^{p-1}\}$  
 	if $\varepsilon_i  = -1$ and
 	$w_i = \{\epsilon, a, \dots, a^{q-1}\}$ 
 	if $\varepsilon_i = +1$. 
 	The language $L$ of such normal forms is clearly 
 	regular.
 	For a bijection between the language $L$ and the 
 	group $BS(p,q)$ we take the canonical mapping: 
 	$\pi : L \rightarrow G$.  
 	The right multiplications by $a$ and 
 	$a^{-1}$ are as follows: 
 	$
 	w_\ell t^{\varepsilon_\ell} \dots
 	w_1 t^{\varepsilon_1} a^k \xrightarrow {\times a}
 	w_\ell t^{\varepsilon_\ell} \dots  
 	w_1 t^{\varepsilon_1} a^{k+1}$,
 	$   
 	w_\ell t^{\varepsilon_\ell} \dots
 	w_1 t^{\varepsilon_1} a^k \xrightarrow {\times a^{-1}}
 	w_\ell t^{\varepsilon_\ell} \dots  
 	w_1 t^{\varepsilon_1} a^{k-1}.
 	$
 	Let $k = m q + r$ for $m \in \mathbb{Z}$ and 
 	$r \in \{0,1,\dots,q-1\}$. 
 	The right multiplication by 
 	$t$ is as follows (different cases are
 	considered separately):    
 	\begin{itemize}
 		\item{if $r \neq 0$, then  
 			$w_\ell t^{\varepsilon_\ell} \dots
 			w_1 t^{\varepsilon_1} a^k \xrightarrow {\times t}
 			w_\ell t^{\varepsilon_\ell} \dots
 			w_1 t^{\varepsilon_1} a^r t a^{mp};$} 
 		\item{if $r = 0$, $\ell \geqslant 1$ and 
 			$\varepsilon_1 = +1$, then
 			$$w_\ell t^{\varepsilon_\ell} \dots
 			w_1 t^{\varepsilon_1} a^k \xrightarrow {\times t} 
 			w_\ell t^{\varepsilon_\ell} \dots
 			w_1 t t a^{mp};
 			$$
 		}  
 		\item{if $r=0$, $\ell \geqslant 1$ and 
 			$\varepsilon_1 = -1$, then 
 			$$  
 			w_\ell t^{\varepsilon_\ell} \dots
 			w_1 t^{\varepsilon_1} a^k \xrightarrow {\times t}
 			w_\ell t^{\varepsilon_\ell} \dots
 			w_2 t^{\varepsilon_2} w_1 a^{mp}; 
 			$$
 		} 
 		\item{if $r=0$ and $\ell=0$, then
 			$a^k \xrightarrow {\times t} t a^{mp}$. 
 		}  
 	\end{itemize}
 	Let $k= n p  + s$ for $n \in \mathbb{Z}$ and $s \in  \{0,1,\dots,p-1\}$. The right multiplication by 
 	$t^{-1}$ is as follows: 
 	\begin{itemize} 
 		\item{if $s \neq 0$, then 
 			$w_\ell t^{\varepsilon_\ell} \dots 
 			w_1 t^{\varepsilon_1}a^k \xrightarrow{\times t^{-1}} w_\ell t^{\varepsilon_\ell} \dots 
 			w_1 t^{\varepsilon_1} a^s t^{-1} a^{nq}$;}	
 		\item{if $s=0$, $\ell \geqslant 1$ and 
 			$\varepsilon_1 = +1$, then 
 			$$w_\ell t^{\varepsilon_\ell} \dots 
 			w_1 t^{\varepsilon_1} a^k \xrightarrow{\times 
 				t^{-1}}
 			w_\ell t^{\varepsilon_\ell} \dots 
 			w_2 t^{\varepsilon_2} w_1 a^{nq};$$}
 		\item{if $s=0$, $\ell \geqslant 1$ and $\varepsilon_1 = -1$, then $$w_\ell t^{\varepsilon_\ell} \dots 
 			w_1 t^{\varepsilon_1}a^k \xrightarrow{\times 
 				t^{-1}} w_\ell t^{\varepsilon_\ell} \dots
 			w_1 t^{-1}t^{-1} a^{nq};$$}  
 		\item{if $s=0$ and $\ell = 0$, then $a^k 
 			\xrightarrow{\times t^{-1}} t^{-1} a^{nq}$.}  
 	\end{itemize}	
 	It can be seen that each of the right multiplications 
 	by $a,a^{-1},t$ and $t^{-1}$ shown above 
    is polynomial--time computable. 	
 	Therefore, $\pi : L \rightarrow BS(p,q) $ is 
 	a $\REG$--Cayley polynomial--time computable 
 	representation. Moreover, for  
 	$\pi: L \rightarrow BS(p,q)$ and a natural mapping 
 	$\alpha: \Sigma \rightarrow BS(p,q)$, for 
 	which $\alpha(a) = a$, $\alpha(a^{-1}) = a^{-1}$, 
 	$\alpha (t) = t$ and $\alpha(t^{-1})=t^{-1}$, 
 	the \distfun\ $h_{\pi,\alpha}$ 
 	vanishes.
 \end{proof} 
 \begin{remark}	
 	It follows from the metric 
 	estimates for the Baumslag--Solitar group 
 	$BS(p,q)$ obtained by
 	Burillo and the second author \cite{BurilloElder14} 
 	that the 
 	$\REG$--Cayley polynomial--time 
 	computable representation
 	$\pi : L \rightarrow BS(p,q)$ from the  
 	proof of Theorem \ref{baumslagthm}
 	does not have quasigeodesic normal form; 
 	see also a proof of the analogous fact  in \cite[p.~317]{ElderTabackCgraph}.  	
 \end{remark}


\section{Conclusion} 
\label{sec_conclusion}    
In this paper we introduced the notion 
of a 
\ClinearC\ 
and a $\mathcal{C}$--Cayley polynomial--time 
computable group which  extend
the notion of a Cayley automatic group 
introduced by Kharlampovich, Khoussainov and 
Miasnikov. We proved some algorithmic and 
closure properties for these groups, 
and showed examples. 
We analysed behaviour of the \distfun\ for  
\linearC\ and $\REG$--\polyC\ representations.  
For future work we plan to focus  
on the classes of \linearC\ and 
$\REG$--\polyC\ groups.

\section*{Acknowledgements} 
 The second author is supported by Australian Research Council grant DP160100486.  
 The authors thank the anonymous reviewers for their helpful feedback and careful proofreading.


\bibliographystyle{elsarticle-num}

\bibliography{cayleypolbibfile}

\end{document}